\documentclass[]{interact}

\usepackage{epstopdf}
\usepackage{subfigure}

\usepackage{natbib}
\usepackage{mathrsfs}
\usepackage{dsfont}
\usepackage[pagebackref=false,colorlinks,linkcolor=black,citecolor=blue]{hyperref}
\bibpunct[, ]{(}{)}{;}{a}{}{,}
\renewcommand\bibfont{\fontsize{10}{12}\selectfont}
\makeatother
\makeatletter
\def\namedlabel#1#2{\begingroup
   \def\@currentlabel{#2}%
   \label{#1}\endgroup}
\makeatother
\theoremstyle{plain}
\newtheorem{theorem}{Theorem}[section]
\newtheorem{lemma}[theorem]{Lemma}
\newtheorem{corollary}[theorem]{Corollary}
\newtheorem{proposition}[theorem]{Proposition}

\theoremstyle{definition}
\newtheorem{definition}[theorem]{Definition}
\newtheorem{example}[theorem]{Example}

\theoremstyle{remark}
\newtheorem{remark}{Remark}

\begin{document}

\title{The hull-kernel topology on residuated lattices}

\author{
\name{Saeed Rasouli\textsuperscript{a}\thanks{CONTACT Saeed Rasouli. Email: srasouli@pgu.ac.ir} and Amin Dehghani\textsuperscript{b}}
\affil{\textsuperscript{a,b}Department of Mathematics, Persian Gulf University, Bushehr, 75169, Iran}
}

\maketitle
\begin{abstract}
The notion of hull-kernel topology on a collection of prime filters in a residuated lattice is introduced and investigated. It is observed that any collection of prime filters is a $T_0$ topological space under the hull-kernel and the dual hull-kernel topologies. It is proved that any collection of prime filters is a $T_1$ space if and only if it is an antichain, and it is a Hausdorff space if and only if it satisfies a certain condition. Some characterizations in which maximal filters forms a Hausdorff space are given. At the end, it is focused on the space of minimal prim filters, and it is shown that this space is a totally disconnected Hausdorff space. This paper is closed by a discussion abut the various forms of compactness and connectedness of this space.
\end{abstract}

\section{Introduction}

\cite{sto} established a correspondence between the category of distributive lattices and a certain category of topological spaces. He observed that the set of prime ideals of a Boolean algebra forms a topological space in a natural way. A form of Stone's representation theorem for Boolean algebras states that every Boolean algebra is isomorphic to the Boolean algebra of clopen sets of a non-empty compact totally disconnected Hausdorff space (Stone space). This isomorphism forms a category-theoretic duality between the categories of Boolean algebras and Stone spaces. \cite{kis} and \cite{hen} investigated on commutative semigroups with $0$ and commutative rings, respectively, via the space of minimal prime ideals. The space of minimal prime ideals for distributive lattices with $0$ is discussed in \cite{spe,spe2} where the author presented a characterization for various forms of compactness and connectedness. The space of minimal prime ideals of a ring without nilpotent elements is discussed in \cite{tha} and an interesting characterization of Baer rings is given. In \cite{leus0}, the author introduced the notion of Stone topology on the prime and maximal filters in a BL algebra and investigated them. He showed that the set of prime filters is a compact $T_0$ space and the set of maximal prime filters is a compact Hausdorff space under this topology. Recently, the space of minimal prime ideals of a poset is investigated by \cite{mund}.

In this work we study the hull-kernel topology on a collection of prime filters in a residuated lattice. This paper is organized in four sections as follow: In Section \ref{sec2}, some definitions and facts about residuated lattices and topology are recalled, which are used in the next sections. Moreover, notions of maximal, prime, and minimal prime filters are introduced and some characterizations of them are given. In Section \ref{sec3}, the hull and kernel operators on a collection of filters in a residuated lattices are introduced and a Galois connection is established between them (Theorem \ref{clotopo}). It is observed that any collection of prime filters is a $T_0$ topological space under the hull-kernel and the dual hull-kernel topologies (Theorem \ref{t0spacecon}). It is proved that any collection of prime filters is a $T_1$ space if and only if it is an antichain (Theorem \ref{t1spa}), and it is a Hausdorff space if and only if it satisfies a certain condition (Theorem \ref{hauspro}). This section further concentrates on maximal filters and some characterizations in which this collection forms a Hausdorff space. Section \ref{sec4} is focused on the space of minimal prim filters and it is shown that this space is a totally disconnected Hausdorff space (Corollary \ref{minspapro}). This section is closed by a discussion abut the various forms of compactness and connectedness of this space.
\section{Preliminaries}\label{sec2}

In this section we recall some definitions and facts about residuated lattices and topology.
\subsection{Residuated lattices}

This subsection is devoted to recall some definitions, properties and results relative to residuated lattices, which will be used
in the following. The results in the this section are original, excepting those that we cite from other papers.

An algebra $\mathfrak{A}=(A;\vee,\wedge,\odot,\rightarrow,0,1)$ is called a \textit{residuated lattice} if $\ell(\mathfrak{A})=(A;\vee,\wedge,0,1)$ is a bounded lattice, $(A;\odot,1)$ is a commutative monoid and $(\odot,\rightarrow)$ is an adjoint pair. A residuated lattice $\mathfrak{A}$ is called a \textit{MTL algebra} if satisfying the \textit{pre-linearity condition} (denoted by \ref{prel}):
\begin{enumerate}
\item [$(prel)$ \namedlabel{prel}{$(prel)$}] $(x\rightarrow y)\vee(y\rightarrow x)=1$, for all $x,y\in A$.
\end{enumerate}

In a residuated lattice $\mathfrak{A}$, for any $a\in A$, we put $\neg a:=a\rightarrow 0$. It is well-known that the class of residuated lattices is equational \citep{idz}, and so it forms a variety. The properties of residuated lattices were presented in \cite{gal}. For a survey of residuated lattices we refer to \cite{jip}.
\begin{remark}\label{resproposition}\citep[Proposition 2.2]{jip}
Let $\mathfrak{A}$ be a residuated lattice. The following conditions are satisfied for any $x,y,z\in A$:
\begin{enumerate}
  \item [$r_{1}$ \namedlabel{res1}{$r_{1}$}] $x\odot (y\vee z)=(x\odot y)\vee (x\odot z)$;
  \item [$r_{2}$ \namedlabel{res2}{$r_{2}$}] $x\vee (y\odot z)\geq (x\vee y)\odot (x\vee z)$.
  \end{enumerate}
\end{remark}
\begin{example}\label{rex2}
Let $A_6=\{0,a,b,c,d,1\}$ be a lattice whose Hasse diagram is below (see Figure \ref{graph6}).  Define $\odot$ and $\rightarrow$ on $A_7$ as follows:
\begin{eqnarray*}
\begin{array}{l|llllll}
  \odot & 0 & a & b & c & d &  1   \\ \hline
    0   & 0 & 0 & 0 & 0 & 0 &  0  \\
    a   & 0 & a & a & 0 & a &  a  \\
    b   & 0 & a & a & 0 & a &  b  \\
    c   & 0 & 0 & 0 & c & c &  c  \\
    d   & 0 & a & a & c & d &  d  \\
    1   & 0 & a & b & c & d &  1
  \end{array}& \hspace{1cm} &
  \begin{array}{l|llllll}
  \rightarrow & 0 & a & b & c & d &  1   \\ \hline
    0         & 1 & 1 & 1 & 1 & 1 &  1  \\
    a         & c & 1 & 1 & c & 1 &  1  \\
    b         & c & d & 1 & c & 1 &  1  \\
    c         & b & b & b & 1 & 1 &  1  \\
    d         & 0 & b & b & c & 1 &  1  \\
    1         & 0 & a & b & c & d &  1
  \end{array}
\end{eqnarray*}
\begin{figure}[h]
\centering
\includegraphics[scale=.1]{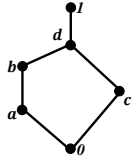}
\caption{The Hasse diagram of $\mathfrak{A}_6$.}
\label{graph6}
\end{figure}
Routine calculation shows that  $\mathfrak{A}_6=(A_6;\vee,\wedge,\odot,\rightarrow,0,1)$ is a residuated lattice.
\end{example}

Let $\mathfrak{A}$ be a residuated lattice. A non-void subset $F$ of $A$ is called a \textit{filter} of $\mathfrak{A}$ if $x,y\in F$ implies $x\odot y\in F$ and $x\vee y\in F$ for any $x\in F$ and $y\in A$. The set of filters of $\mathfrak{A}$ is denoted by $\mathscr{F}(\mathfrak{A})$. A filter $F$ of $\mathfrak{A}$ is called \textit{proper} if $F\neq A$. Clearly, $F$ is a proper filter if and only if $0\notin F$. For any subset $X$ of $A$ the \textit{filter of $\mathfrak{A}$ generated by $X$} is denoted by $\mathscr{F}(X)$. For each $x\in A$, the filter generated by $\{x\}$ is denoted by $\mathscr{F}(x)$ and called \textit{principal filter}. The set of principal filters is denoted by $\mathscr{PF}(\mathfrak{A})$. Let $\mathcal{F}$ be a collection of filters of $\mathfrak{A}$. Set $\veebar \mathcal{F}=\mathscr{F}(\cup \mathcal{F})$. It is well-known that $(\mathscr{F}(\mathfrak{A});\cap,\veebar,\textbf{1},A)$ is a frame and so it is a complete Heyting algebra.
\begin{example}\label{fex2}
Consider the residuated lattice $\mathfrak{A}_6$ from Example \ref{rex2}. Then $\mathscr{F}(\mathfrak{A}_6)=\{F_1=\{1\},F_2=\{d,1\},F_3=\{a,b,d,1\},F_4=\{c,d,1\},F_5=A_6\}$.
\end{example}

The following remark has a routine verification.
\begin{remark}\label{genfilprop}
Let $\mathfrak{A}$ be a residuated lattice, $F$ be a filter and $X$ be a subset of $A$. The following assertions hold for any $x,y\in A$:
\begin{enumerate}
  \item  [$(1)$ \namedlabel{genfilprop1}{$(1)$}] $\mathscr{F}(F,x):=F\veebar \mathscr{F}(x)=\{a\in A|f\odot x^n\leq a,~f\in F\}$;
  \item  [$(2)$ \namedlabel{genfilprop2}{$(2)$}] $x\leq y$ implies $\mathscr{F}(F,y)\subseteq \mathscr{F}(F,x)$.
  \item  [$(3)$ \namedlabel{genfilprop3}{$(3)$}] $\mathscr{F}(F,x)\cap \mathscr{F}(F,y)=\mathscr{F}(F,x\vee y)$;
  \item  [$(4)$ \namedlabel{genfilprop4}{$(4)$}] $\mathscr{F}(F,x)\veebar \mathscr{F}(F,y)=\mathscr{F}(F,x\odot y)$;
  \item  [$(5)$ \namedlabel{genfilprop5}{$(5)$}] $\mathscr{PF}(\mathfrak{A})$ is a sublattice of $\mathscr{F}(\mathfrak{A})$;
  \item  [$(6)$ \namedlabel{genfilprop6}{$(6)$}] if $A=\mathscr{F}(X)$, then $A=\mathscr{F}(Y)$ for some finite subset $Y\subseteq X$.
\end{enumerate}
\end{remark}

A proper filter of a residuated lattice $\mathfrak{A}$ is called \textit{maximal} if it is a maximal element in the set of all proper filters. The set of all maximal filters of $\mathfrak{A}$ is denoted by $Max(\mathfrak{A})$. A proper filter $P$ of $\mathfrak{A}$ is called \textit{prime} if $F_1\cap F_2=P$ implies $F_1=P$ or $F_2=P$ for any $F_1,F_2\in \mathscr{F}(\mathfrak{A}_{\nu})$. The set of all prime filters of $\mathfrak{A}$ is denoted by $Spec(\mathfrak{A})$.  It is obvious that $Max(\mathfrak{A})\subseteq Spec(\mathfrak{A})$. By Zorn's lemma follows that any proper filter is contained in a maximal filter and so in a prime filter.
\begin{proposition}\label{preqpro}
Let $\mathfrak{A}$ be a residuated lattice and $P$ be a proper filter of $\mathfrak{A}$. The following assertions are equivalent:
\begin{enumerate}
\item  [(1) \namedlabel{preqpro1}{(1)}] $P$ is prime;
\item  [(2) \namedlabel{preqpro2}{(2)}] $F_1\cap F_2\subseteq P$ implies $F_1\subseteq P$ or $F_2\subseteq P$ for any $F_1,F_2\in \mathscr{F}(\mathfrak{A}_{\nu})$;
\item  [(3) \namedlabel{preqpro3}{(3)}] $x\vee y\in P$ implies $x\in P$ or $y\in P$ for any $x,y\in A$.
\end{enumerate}
\end{proposition}

A non-empty subset $\mathscr{C}$ of $\mathfrak{A}$ is called \textit{$\vee$-closed} if it is closed under the join operation, i.e $x,y\in \mathscr{C}$ implies $x\vee y\in \mathscr{C}$.
\begin{remark}\label{primclos}
It is obvious that a filter $P$ is prime if and only if $P^{c}$ is $\vee$-closed. Also, if $\mathscr{P}\subseteq Spec(\mathfrak{A})$, then $(\cup \mathscr{P})^{c}$ is a $\vee$-closed subset of $\mathfrak{A}$.
\end{remark}

The following result is an easy consequence of Zorn's lemma.
\begin{lemma}\label{0prfiltth}
If $\mathscr{C}$ is a $\vee$-closed subset of $\mathfrak{A}$ which does not meet the filter $F$, then $\mathscr{C}$ is contained in a $\vee$-closed subset $\textsf{C}$ which is maximal with respect to the property of not meeting $F$.
\end{lemma}

The following important result is proved for pseudo-BL algebras \cite[Theorem 4.28]{din0}; however, it can be proved without difficulty in all residuated lattices.
\begin{theorem}\label{prfilth}(Prime filter theorem)
If $\mathscr{C}$ is a $\vee$-closed subset of $\mathfrak{A}$ which does not meet the filter $F$, then $F$ is contained in a filter $P$ which is maximal with respect to the property of not meeting $\mathscr{C}$; furthermore $P$ is prime.
\end{theorem}
\begin{corollary}\label{intprimfilt}
Let $F$ be a filter of a residuated lattice $\mathfrak{A}$ and $X$ be a subset of $A$. The following assertions hold:
\begin{enumerate}
\item  [(1) \namedlabel{intprimfilt1}{(1)}]  If $X\nsubseteq F$, there exists a prime filter $P$ such that $F\subseteq P$ and $X\nsubseteq P$;
\item  [(2) \namedlabel{intprimfilt2}{(2)}] $\mathscr{F}(X)=\bigcap \{P\in Spec(\mathfrak{A})|X\subseteq P\}$.
\end{enumerate}
\end{corollary}
\begin{proof}
\begin{enumerate}
  \item [\ref{intprimfilt1}:] Let $x\in X-F$. By taking $\mathscr{C}=\{x\}$ it follows by Theorem \ref{prfilth}.
  \item [\ref{intprimfilt2}:] Set $\sigma_{X}=\{P\in Spec(\mathfrak{A})|X\subseteq P\}$. Obviously, we have $\mathscr{F}(X)\subseteq \bigcap \sigma_{X}$. Now let $a\notin \mathscr{F}(X)$. By \ref{intprimfilt1} follows that there exits a prime filter $P$ containing $\mathscr{F}(X)$ such that $a\notin P$. It shows that $a\notin \bigcap \sigma_{X}$.
\end{enumerate}
\end{proof}

Let $\mathfrak{A}$ be a residuated lattice and $X$ be a subset of $A$. A prime filter $P$ is called a \textit{minimal prime filter belonging to $X$} or $X$-\textit{minimal prime filter} if $P$ is a minimal element in the set of prime filters containing $X$. The set of $X$-minimal prime filters of $\mathfrak{A}$ is denoted by $Min_{X}(\mathfrak{A})$. A prime filter $P$ is called a \textit{minimal prime} if $P\in Min_{\{1\}}(\mathfrak{A})$. The set of minimal prime filters of $\mathfrak{A}$ is denoted by $Min(\mathfrak{A})$.

In following we give an important characterization for minimal prime filters.
\begin{theorem}\label{1mineq}(Minimal prime filter theorem)
Let $\mathfrak{A}$ be a residuated lattice and $F$ be a filter of $\mathfrak{A}$. A subset $P$ of $A$ is an $F$-minimal prime filter if and only if $P^{c}$ is a $\vee$-closed subset of $\mathfrak{A}$ which it is maximal with respect to the property of not meeting $F$.
\end{theorem}
\begin{proof}
Let $P$ be a subset of $A$ such that $P^{c}$ is a $\vee$-closed subset of $\mathfrak{A}$
which is maximal w.r.t the property of not meeting $F$. By Proposition \ref{prfilth} there exists a prime filter $Q$ such that $Q$ not meeting $P^{c}$ and so $Q\subseteq P$. By \textsc{Remark} \ref{primclos}, $Q^{c}$ is a $\vee$-closed subset of $\mathfrak{A}$ and by hypothesis we have $P^{c}\subseteq Q^{c}$ and $Q^{c}\cap F=\emptyset$. So by maximality of $P^{c}$ we deduce that $P^{c}=Q^{c}$ and it means that $P=Q$. It shows that $P$ is a prime filter and moreover it shows that $P$ is an $F$-minimal prime filter.

Conversely, let $P$ be an $F$-minimal prime filter of $\mathfrak{A}$. By \textsc{Remark} \ref{primclos}, $P^{c}$ is a $\vee$-closed subset of $\mathfrak{A}$ such that $P^{c}\cap F=\emptyset$. By using Lemma \ref{0prfiltth} we can obtain a $\vee$-closed subset $\mathscr{C}$ of $\mathfrak{A}$ such that it is maximal with respect to the property of not meeting $F$. By case just proved, $\mathscr{C}'$ is an $F$-minimal prime filter such that $\mathscr{C}'\cap P^{c}=\emptyset$ and it implies $\mathscr{C}'\subseteq P$. By hypothesis $\mathscr{C}=P^{c}$ and it shows that $P^{c}$ is a $\vee$-closed subset of $\mathfrak{A}$ such that it is maximal with respect to the property of not meeting $F$.
\end{proof}
\begin{corollary}\label{primeminimal}
Let $\mathfrak{A}$ be a residuated lattice, $X$ be a subset of $A$ and $P$ be a prime filter containing $X$. Then there exists an $X$-minimal prime filter contained in $P$.
\end{corollary}
\begin{proof}
By \textsc{Remark} \ref{primclos}, $P^{c}$ is a $\vee$-closed subset of $\mathfrak{A}$ such that $P^{c}\cap \mathscr{F}(X)=\emptyset$. By using Lemma \ref{0prfiltth} we can obtain a $\vee$-closed subset $\mathscr{C}$ of $\mathfrak{A}$ containing $P^{c}$ such that it is maximal with respect to the property of not meeting $\mathscr{F}(X)$. By Theorem \ref{1mineq}, $\mathscr{C}'$ is an $\mathscr{F}(X)$-minimal prime filter which it is contained in $P$.
\end{proof}

The following corollary should be compared with Corollary \ref{intprimfilt}.
\begin{corollary}\label{mininters}
Let $F$ be a filter of a residuated lattice $\mathfrak{A}$ and $X$ be a subset of $A$. The following assertions hold:
\begin{enumerate}
\item  [(1) \namedlabel{mininters1}{(1)}]  If $X\nsubseteq F$, there exists an $F$-minimal prime filter $\mathfrak{m}$ such that $X\nsubseteq \mathfrak{m}$;
\item  [(2) \namedlabel{mininters2}{(2)}] $\mathscr{F}(X)=\bigcap Min_{X}(\mathfrak{A})$.
\end{enumerate}
\end{corollary}
\begin{proof}
\begin{enumerate}
  \item [\ref{mininters1}:] It is a direct consequence of Corollary \ref{intprimfilt}\ref{intprimfilt1} and Corollary \ref{primeminimal}.
  \item [\ref{mininters2}:] Set $\sigma_{X}=\{P\in Spec(\mathfrak{A})|X\subseteq P\}$. By Corollary \ref{intprimfilt}\ref{intprimfilt2}, it is sufficient to show that $\bigcap Min_{X}(\mathfrak{A})=\bigcap \sigma_{X}$. It is obvious that $\bigcap \sigma_{X}\subseteq \bigcap Min_{X}(\mathfrak{A})$. Otherwise, let $a\in \bigcap Min_{X}(\mathfrak{A})$ and $P$ be an arbitrary element of $\sigma_{X}$. By Corollary \ref{primeminimal} there exists an $X$-minimal prime filter $\mathfrak{m}$ contained in $P$. Hence, $a\in \mathfrak{m}\subseteq P$ and it states that $\bigcap Min_{X}(\mathfrak{A})\subseteq \bigcap \sigma_{X}$.
\end{enumerate}
\end{proof}

Let $\mathfrak{A}$ be a residuated lattice and $F$ be a filter of $\mathfrak{A}$. For any subset $X$ of $A$ \textit{the coannihilator of $X$ belonging to $F$ in $\mathfrak{A}$} is denoted by $(F:X)$ and defined as follow:
\[(F:X)=\{a\in A|x\vee a\in F,\forall x\in X\}.\]
If $X=\{x\}$, we write $(F:x)$ instead of $(F:X)$ and in case $F=\{1\}$, we write $X^{\perp}$ instead of $(F:X)$.
\begin{proposition}\label{ccapspecx}
Let $\Pi$ be a collection of prime filters in a residuated lattice $\mathfrak{A}$. Then, for any subset $X$ of $A$, we have
\[(\bigcap \Pi:X)=\bigcap\{P\in \Pi|~X\nsubseteq P\}.\]
\end{proposition}
\begin{proof}
Set $\Phi=\{P\in \Pi|~X\nsubseteq P\}$. Let $a\in (\bigcap \Pi:X)$ and $P\in \Phi$. Consider $x\in X\setminus P$. Since $x\vee a\in\bigcap \Pi \subseteq P$ so $a\in P$. It states that $a\in \bigcap\Phi$.

Conversely, let $a\in \bigcap\Phi$ and $x\in X$. Obviously, for any $P\in \Pi$ we have $x\vee a\in P$ and it states that $x\vee a\in \bigcap \Pi$. So $a\in (\bigcap \Pi:X)$.
\end{proof}
\begin{corollary}
Let $\mathfrak{A}$ be a residuated lattice. Then, for any subset $X$ of $A$, we have
\[X^{\perp}=\bigcap\{P\in Spec(\mathfrak{A})|~X\nsubseteq P\}=\bigcap\{\mathfrak{m}\in Min(\mathfrak{A})|~X\nsubseteq P\}.\]
\end{corollary}
\begin{proof}
It follows by Corollary \ref{intprimfilt}\ref{intprimfilt2} and \ref{mininters}, and Proposition \ref{ccapspecx}.
\end{proof}

Let $\mathfrak{A}$ be a residuated lattice and $F$ be a filter of $\mathfrak{A}$. We set
\[D_{F}(P)=\{a\in A|(F:a)\nsubseteq P\}.\]
If $F=\{1\}$, we write $D(P)$ instead of $D_{F}(P)$.
\begin{proposition}\label{lemdivi}
 Let $\mathfrak{A}$ be a residuated lattice, $F$ be a filter of $\mathfrak{A}$ and $P$ be a prime filter of $\mathfrak{A}$. The following assertions hold:
\begin{enumerate}
  \item  [$(1)$ \namedlabel{lemdivi1}{$(1)$}] $F\subseteq D_{F}(P)=\cup_{x\notin P}(F:x)$;
  \item  [$(2)$ \namedlabel{lemdivi2}{$(2)$}] if $Q$ is a prime filter of $\mathfrak{A}$ containing $P$, then $D_{F}(Q)\subseteq D_{F}(P)$;
  \item  [$(3)$ \namedlabel{lemdivi3}{$(3)$}] if $P$ contains $F$, then $D_{F}(P)\subseteq P$;
%
\end{enumerate}
 \end{proposition}
\begin{proof}
We only prove the case \ref{lemdivi3}, because the other cases can be proved in a routine way.
\begin{enumerate}
\item [\ref{lemdivi3}:] Let $P$ contains $F$. Since $(F:a)\subseteq P$ for any $a\notin P$ so it follows by \ref{lemdivi1}.
%
%
\end{enumerate}
\end{proof}
The concept of a $pm$-lattice is introduced by \cite{paw} as a bounded distributive lattice in which each prime ideal is contained in a unique maximal prime ideal. In the following theorem a characterization of $pm$-lattices are given.
\begin{theorem}\citep{paw1}
Let $\mathfrak{A}$ be a bounded distributive lattice. Then $\mathfrak{A}$ is a $pm$-lattice if and only if for any two distinct maximal ideals $M_1$ and $M_2$, there exist $a_1\notin M_1$ and $a_2\notin M_2$ such that $a_1\vee a_2=0$.
\end{theorem}

In the following definition we generalize the dual of this equivalent assertion.
\begin{definition}
Let $\mathfrak{A}$ be a residuated lattice, $\Pi$ be a collection of prime filters and $F$ be a filter containing in $\bigcap \Pi$. $\Pi$ is called \textit{$F$-closed} if for any distinct elements $P_1$ and $P_2$ in $\Pi$ there exist $a_1\notin P_1$ and $a_2\notin P_2$ such that $a_1\vee a_2\in F$.
\end{definition}
\begin{proposition}\label{pmprop}
Let $\mathfrak{A}$ be a residuated lattice, $\Pi$ be a collection of prime filters and $F$ be a filter containing in $\bigcap \Pi$.  The following assertions hold:
 \begin{enumerate}
   \item [(1) \namedlabel{pmprop1}{(1)}] $\Pi$ is $F$-closed;
   \item [(2) \namedlabel{pmprop2}{(2)}] for any $P\in \Pi$, $P$ is the unique element in $\Pi$ containing $D_{F}(P)$.
 \end{enumerate}
\end{proposition}
\begin{proof}
\item \ref{pmprop1}$\Rightarrow$\ref{pmprop2}: Let $P_1$ and $P_2$ be distinct elements in $\Pi$ such that $D_{F}(P_1)\subseteq P_2$. By hypothesis there exits $a_1\notin P_1$ and $a_2\notin P_2$ such that $a_1\vee a_2\in F$. It implies that $a_2\in D_{F}(P_1)\subseteq P_2$; a contradiction.
\item \ref{pmprop2}$\Rightarrow$\ref{pmprop1}: Let $P_1$ and $P_2$ be distinct elements in $\Pi$. So $D_{F}(P_1)\nsubseteq P_2$. Consider $a_2\in D_{F}(P_1)\setminus P_2$. Since $a_2\in D_{F}(P_1)$ so there exists $a_1\notin P_1$ such that $a_1\vee a_2\in F$.
\end{proof}

In the following, for a collection of prime filters $\Pi$ in a residuated lattice $\mathfrak{A}$ we let $S_{\Pi}=\{Q\in Spec(\mathfrak{A})|\bigcap \Pi\subseteq Q\subseteq P,~for~some~P\in \Pi\}$.
\begin{proposition}\label{chauspro}
Let $\mathfrak{A}$ be a residuated lattice, $\Pi$ be a collection of prime filters and $F$ be a filter containing in $\bigcap \Pi$. Consider the following assertions:
\begin{enumerate}
   \item [(1) \namedlabel{chauspro1}{(1)}] $\Pi$ is $F$-closed;
   \item [(2) \namedlabel{chauspro2}{(2)}] any element of $S_{\Pi}$ is contained in a unique element of $\Pi$;
   \item [(3) \namedlabel{chauspro3}{(3)}] $\Pi$ is an antichain.
 \end{enumerate}
 Then \ref{hauspro1} implies \ref{hauspro2} and \ref{hauspro2} implies \ref{hauspro3}. Moreover, \ref{hauspro3} implies \ref{hauspro1} provided that $(x\rightarrow y)\vee (y\rightarrow x)\in F$ for any $x,y\in A$.
\end{proposition}
\begin{proof}
\item \ref{chauspro1}$\Rightarrow$\ref{chauspro2}: Let $Q\in S_{\Pi}$ and $Q\subseteq P$ for some $P\in \Pi$. Since $D_{F}(P)\subseteq Q$ so the result holds by Proposition \ref{pmprop}.
\item \ref{chauspro2}$\Rightarrow$\ref{chauspro3}: It is trivial, since $\Pi\subseteq S_{\Pi}$.

Now, let $(x\rightarrow y)\vee (y\rightarrow x)\in F$ for any $x,y\in A$.
\item \ref{chauspro3}$\Rightarrow$\ref{chauspro1}: Let $P_1$ and $P_2$ be distinct elements in $\Pi$. So there exist $x\in P_1\setminus P_2$ and $y\in P_2\setminus P_1$. It implies that $x\rightarrow y\notin P_1$ and $y\rightarrow x\notin P_2$. Since $(x\rightarrow y)\vee (y\rightarrow x)\in F$ so the result holds.
\end{proof}
\begin{corollary}\label{mtlcloant}
Let $\mathfrak{A}$ be a MTL algebra and $\Pi$ be a collection of prime filters in $\mathfrak{A}$. Then $\Pi$ is $\{1\}$-closed if and only if $\Pi$ is an antichain.
\end{corollary}
\begin{proof}
It is straightforward by Proposition \ref{chauspro}.
\end{proof}


The following corollary is a characterization for minimal prime filters.
\begin{theorem}\label{mincor}
Let $\mathfrak{A}$ be a residuated lattice, $F$ be a filter and $P$ be a prime filter containing $F$. The following assertions are equivalent:
\begin{enumerate}
  \item  [$(1)$ \namedlabel{mincor1}{$(1)$}] $P$ is an $F$-minimal prime filter;
  \item  [$(2)$ \namedlabel{mincor2}{$(2)$}] $P=D_{F}(P)$;
  \item  [$(3)$ \namedlabel{mincor3}{$(3)$}] for any $x\in A$, $P$ contains precisely one of $x$ or $(F:x)$.
\end{enumerate}
\end{theorem}
\begin{proof}
\item [] \ref{mincor1}$\Rightarrow$ \ref{mincor2}: Let $x\in P$. It is easy to check that $\mathscr{C}=(x\vee P^{c})\cup P^{c}$ is a $\vee$-closed subset of $\mathfrak{A}$. By Theorem \ref{1mineq} we obtain that $(x\vee P^{c})\cap F=\mathscr{C}\cap F\neq \emptyset$. Assume that $a\in (x\vee P^{c})\cap F$. So there exists $y\notin P$ such that $x\vee y=a\in F$ and it means that  $x\in D_{F}(P)$. The converse inclusion is evident by Proposition \ref{lemdivi}\ref{lemdivi3}.
\item [] \ref{mincor2}$\Rightarrow$ \ref{mincor3}: It is straightforward by Proposition \ref{lemdivi}\ref{lemdivi1}.
\item [] \ref{mincor3}$\Rightarrow$ \ref{mincor1}: Let $Q$ be a prime filter containing $F$ such that $Q\subseteq P$. Consider $x\in P$. So $(F:x)\nsubseteq P$ and it implies that $x\in D_{F}(P)\subseteq D_{F}(Q)\subseteq Q$ and this shows that $P=Q$.
\end{proof}
\begin{lemma}\label{minmdfp}
Let $\mathfrak{A}$ be a residuated lattice, $F$ be a filter and $P$ be a prime filter of $\mathfrak{A}$. Then any $D_{F}(P)$-minimal prime filter is contained in $P$.
\end{lemma}
\begin{proof}
Let $\mathfrak{m}$ be a $D_{F}(P)$-minimal prime filter and $x\in \mathfrak{m}\setminus P$.  By Theorem \ref{mincor} we have $x\in D_{D_{F}(P)}(\mathfrak{m})$ and it implies that $x\vee y\in D_{F}(P)$ for some $y\notin \mathfrak{m}$. So there exists $z\notin P$ such that $z\in (F:x\vee y)$. It is easy to see that $y\in (F:x\vee z)\subseteq D_{F}(P)\subseteq \mathfrak{m}$ and it leads us to a contradiction.
\end{proof}
\begin{lemma}\label{minpridfp}
Let $\mathfrak{A}$ be a residuated lattice, $F$ be a filter and $P$ be a prime filter. We have
\[Min_{D_{F}(P)}(\mathfrak{A})=\{\mathfrak{m}|\mathfrak{m}\in Min_{F}(\mathfrak{A}),~\mathfrak{m}\subseteq P\}.\]
\end{lemma}
\begin{proof}
Set $\mu=\{\mathfrak{m}|\mathfrak{m}\in Min_{F}(\mathfrak{A}),~\mathfrak{m}\subseteq P\}$. If $\mathfrak{m}\in \mu$, then by Proposition \ref{lemdivi}\ref{lemdivi2} and Theorem \ref{mincor} we obtain that $D_{F}(P)(\mathfrak{A})\subseteq D_{F}(\mathfrak{m})=\mathfrak{m}$. So $\mathfrak{m}$ is a prime filter containing $D_{F}(P)(\mathfrak{A})$. Assume that $\mathfrak{w}$ is a prime filter containing $D_{F}(P)(\mathfrak{A})$ such that $\mathfrak{w}\subseteq \mathfrak{m}$. Applying Proposition \ref{lemdivi}(\ref{lemdivi2} and \ref{lemdivi3}) and Theorem \ref{mincor}, we have $\mathfrak{m}=D_{F}(\mathfrak{m})\subseteq D_{F}(\mathfrak{w})\subseteq \mathfrak{w}$. It follows that $\mathfrak{w}=\mathfrak{m}$. Hence, $\mathfrak{m}\in Min_{D_{F}(P)(\mathfrak{A})}(\mathfrak{A})$.

Conversely, let $\mathfrak{m}\in Min_{D_{F}(P)(\mathfrak{A})}(\mathfrak{A})$. By Proposition \ref{lemdivi}\ref{lemdivi1} follows that $F\subseteq \mathfrak{m}$ and by Lemma \ref{minmdfp} follows that $\mathfrak{m}\subseteq P$. Suppose that $\mathfrak{w}$ is a prime filter containing $F$ such that $\mathfrak{w}\subseteq \mathfrak{m}$. Applying Proposition \ref{lemdivi}\ref{lemdivi2}, it shows that $D_{F}(P)\subseteq D_{F}(\mathfrak{w})\subseteq \mathfrak{w}$. Therefore, $\mathfrak{w}$ is a prime filter containing $D_{F}(P)$ and so $\mathfrak{w}=\mathfrak{m}$. It shows that $\mathfrak{m}$ is an $F$-minimal prime filter and so $\mathfrak{m}\in \mu$.
\end{proof}
\begin{corollary}\label{dfmpr}
Let $\mathfrak{A}$ be a residuated lattice, $F$ be a filter and $P$ be a prime filter. We have
\[D_{F}(P)=\bigcap \{\mathfrak{m}|\mathfrak{m}\in Min_{F}(\mathfrak{A}),~\mathfrak{m}\subseteq P\}.\]
\end{corollary}
\begin{proof}
It follows by Corollary \ref{mininters} and Lemma \ref{minpridfp}.
\end{proof}
\begin{corollary}\label{dfppr}
Let $\mathfrak{A}$ be a residuated lattice, $F$ be a filter and $P$ be a prime filter. We have
\[D_{F}(P)=\bigcap \{Q\in Spec(\mathfrak{A})|F\subseteq Q\subseteq P\}.\]
\end{corollary}
\begin{proof}
Set $\mu=\{\mathfrak{m}|\mathfrak{m}\in Min_{F}(\mathfrak{A}),~\mathfrak{m}\subseteq P\}$ and $\sigma=\{Q\in Spec(\mathfrak{A})|F\subseteq Q\subseteq P\}$. Since $\mu\subseteq \Sigma$ so we have $\bigcap \sigma\subseteq \bigcap \mu=D_{F}(P)$. On the other hand, by Proposition \ref{chauspro} follows that $D_{F}(P)\subseteq \bigcap \sigma$ and it shows the result.
\end{proof}


\subsection{Topology}

In this section we recall some definitions, properties and results relative to topology.

A set $A$ with a family $\tau$ of its subsets is called a
topological space, denoted by $(A;\tau)$, if $\emptyset,A\in\tau$ and it is closed under finite intersection and arbitrary unions. The set of all topologies on $A$ will be denoted by $\tau(A)$. It is obvious that $(\mathcal{P}(A);\tau(A))$ is a closed set system. The closure operator associated with the closed set system $(\mathcal{P}(A);\tau(A))$ is denoted by $\tau^{\mathfrak{A}}:\mathcal{P}^2(A)\longrightarrow\mathcal{P}^2(A)$. Thus for all subset $X$ of $\mathcal{P}(A)$, $\tau^{A}(X)=\bigcap\{\tau\in \tau(A)\mid X\subseteq \tau\}$ is the smallest topology of $A$ that contains $X$. $\tau^{A}(X)$ is called \textit{the topology on $A$ generated by $X$} and $X$ is called a \textit{subbase} for the topological space $(A;\tau^{A}(X))$.  It is well-known that if $(A;\tau)$ is a topological space and $X$ is a subset of $\tau$, then $X$ is a subbase of $A$ if and only if $\tau^{A}(X)=\tau$. When there is no ambiguity $\tau^{A}(X)$ shall be denoted by $\tau(X)$.

The members of $\tau$ are called \textit{open sets} of $A$ and their complements are called \textit{closed sets} of $A$. If $X$ is a subset of $A$, the smallest closed set containing $X$ is called the \textit{closure} of $X$ and is denoted by $\overline{X}$, $cl_{\tau}(X)$ or $cl_{A}(X)$. It is well-known that $a\in \overline{X}$ if and only if any open set containing $a$ meets $X$. An open set containing $a\in A$ is called a \textit{neighbourhood} of $a$. The set of all neighbourhoods of $a\in A$ will be denoted by $\mathscr{N}_{a}$. A family $\mathscr{B}\subseteq \tau$ is said to be a \textit{base (base for close sets)} of $\tau$ if each open (close) set is the union (intersection) of members of $\mathscr{B}$. One can see that a family $\mathscr{B}$ of subsets of $A$ is a base for a topological space if for all $N_1,N_2\in\mathscr{B}$ and any element $a\in N_1\cap N_2$ there is $N\in\mathscr{B}$ such that $a\in N\subseteq N_1\cap N_2$ and $\cup\mathscr{B}=A$. This topological space is denoted by $(A;\tau_{\mathscr{B}})$ and it is called the generated topology by $\mathscr{B}$ on $A$ \citep{eng}.

Let $(A;\tau)$ and $(B;\varsigma)$ be topological spaces. A function $f:A\longrightarrow B$ is called a \textit{continuous} map if $f^{-1}(U)\in\tau$ for all $U\in\varsigma$. $f$ is called an \textit{open} map if $f(U)\in\varsigma$ for all $U\in\tau$.
\begin{proposition}\label{comcont}\citep{eng}
Let $(A;\tau)$ and $(B;\varsigma)$ be topological spaces. The following assertions hold:
\begin{enumerate}
  \item [$(1)$ \namedlabel{comcont1}{$(1)$}]$f:A\longrightarrow B$ is continuous if and only if $f^{-1}(Y)\subseteq A$ is closed, for all closed set $Y\subseteq B$;
  \item [$(2)$ \namedlabel{comcont2}{$(2)$}] $f:A\longrightarrow B$ is continuous if and only if $f(\overline{X})\subseteq \overline{f(X)}$, for all $X\subseteq A$;
  \item [$(3)$ \namedlabel{comcont3}{$(3)$}] if $\mathscr{B}$ is a base for $B$, then $f:A\longrightarrow B$ is continuous if and only if $f^{-1}(N)$ is open in $A$, for all $N\in \mathscr{B}$;
  \item [$(4)$ \namedlabel{comcont4}{$(4)$}] if $X$ is a subbase for $B$, then $f:A\longrightarrow B$ is continuous if and only if $f^{-1}(N)$ is open in $A$, for all $N\in X$.
\end{enumerate}
\end{proposition}

  Now, we recalling some separation axioms in a topological space. Let $(A;\tau)$ be a topological space.
  \begin{itemize}
  \item  $A$ is called a $T_0$ \textit{space} if for any pair of distinct points, there exists an open set containing exactly one of these points.
  \item  $A$ is called a $T_1$ \textit{space} if for every pair of distinct points, there exists an open set containing of each point not containing the other.
  \item  $A$ is called a $T_2$ \textit{space (Hausdorff space)} if any two distinct points are separated by open sets.
  \item  $A$ is called a \textit{normal} if any two disjoint closed subsets of $A$ are separated by open sets.
  \item  $A$ is called a $T_4$ \textit{space (normal Hausdorff space )} if it is both $T_1$ and normal.
\end{itemize}

\begin{remark}\label{t1pro}
\begin{enumerate}
  \item $A$ is a $T_1$ space if and only if every point $a\in A$ is the intersection of all its neighbourhoods if and only if for every $a\in A$ the set $\{a\}$ is closed.
  \item In the above definitions we can apply ``basic open set" instead of ``open set".
\end{enumerate}
\end{remark}

  A topological space $(A;\tau)$ is called a \textit{compact space} provided that every family of closed subsets of $A$, which has the finite intersection property (i.e., every finite subfamily has a nonempty intersection), has nonempty intersection.
  \begin{lemma}\label{norhau}\citep{eng}
  If $A$ is a compact space, $B$ is a Hausdorff space and  $f:A\longrightarrow B$ is a continuous map, then $f$ is a close map.
  \end{lemma}

\section{The hull-kernel topology}\label{sec3}

In this section we introduce and study the notion of hull-kernel topology in residuated lattices.
\begin{definition}\label{defhulker}
Let $\mathfrak{A}$ be a residuated lattice and $\Gamma$ be a collection of filters in $\mathfrak{A}$. We have the following definitions:
\begin{itemize}
  \item  [$(1)$ \namedlabel{defhulker1}{$(1)$}] A mapping $h_{\Gamma}:\mathcal{P}(A)\longrightarrow \mathcal{P}(\Gamma)$ defined by $h_{\Gamma}(X)=\{F\in \Gamma|X\subseteq F\}$ for any $X\subseteq A$, is called a $\Gamma$-\textit{hull operator} on $\mathfrak{A}$;
  \item  [$(2)$ \namedlabel{defhulker2}{$(2)$}] a mapping $k_{\Gamma}:\mathcal{P}(\Gamma)\longrightarrow \mathcal{P}(A)$ defined by $k_{\Gamma}(\mathcal{F})=\cap \mathcal{F}$ for any $\mathcal{F}\subseteq \Gamma$, is called a $\Gamma$-\textit{kernel operator} on $\mathfrak{A}$.
\end{itemize}

For any $x\in A$, $h_{\Gamma}(\{x\})$ is denoted by $h_{\Gamma}(x)$ and when there is no ambiguity we drop the subscript $\Gamma$.
\end{definition}
\begin{proposition}\label{1hopro}
Let $\mathfrak{A}$ be a residuated lattice and $\Gamma$ be a collection of filters in $\mathfrak{A}$. The following assertions hold for any $X,Y\subseteq A$ and $\mathcal{F}\subseteq \Gamma$:
\begin{enumerate}
\item  [$(1)$ \namedlabel{1hopro1}{$(1)$}] $X\subseteq k(\mathcal{F})$ if and only if $\mathcal{F}\subseteq h(X)$;
\item  [$(2)$ \namedlabel{1hopro2}{$(2)$}] $h(X)=h(\mathscr{F}(X))$;
\item  [$(3)$ \namedlabel{1hopro3}{$(3)$}] $h(X)=\Gamma$ if and only if $\mathscr{F}(X)\subseteq k(\Gamma)$. In particular, $h(\emptyset)=h(1)=\Gamma$;
\item  [$(4)$ \namedlabel{1hopro4}{$(4)$}] if $A\notin \Gamma$, then $h(A)=h(0)=\emptyset$;
\item  [$(5)$ \namedlabel{1hopro5}{$(5)$}] $h(X)\cup h(Y)\subseteq h(\mathscr{F}(X)\cap\mathscr{F}(Y))$.
\end{enumerate}
\end{proposition}
\begin{proof}
\begin{enumerate}
  \item [\ref{1hopro1}:] If $X\subseteq k(\mathcal{F})$, then for any $F\in \mathcal{F}$ follows that $X\subseteq F$ and so $F\in h(X)$. Conversely, $\mathcal{F}\subseteq h(X)$ implies that $X\subseteq F$ for any $F\in \mathcal{F}$ and it states that $X\subseteq k(\mathcal{F})$.
  \item [\ref{1hopro2}:] It follows by this fact that $\mathscr{F}(X)$ is the least filter containing $X$.
  \item [\ref{1hopro3}:] It follows by \ref{1hopro1} and \ref{1hopro2}.
  \item [\ref{1hopro4}:] It is obvious.
  \item [\ref{1hopro5}:]  Let $X$ and $Y$ be two subsets of $A$ and $F\in h(X)\cup h_(Y)$ for some $F\in \Gamma$. So $X\subseteq F$ or $Y\subseteq F$ and it implies that $\mathscr{F}(X)\subseteq F$ or $\mathscr{F}(Y)\subseteq F$. It states that $\mathscr{F}(X)\cap \mathscr{F}(X)\subseteq F$ and it follows that $F\in h(\mathscr{F}(X)\cap \mathscr{F}(X))$.
\end{enumerate}
\end{proof}

Recalling that, a pair $(f,g)$ is called a \textit{(contravariant or antitone) Galois connection} between posets $\mathscr{A}=(A;\leq)$ and $\mathscr{B}=(B;\preccurlyeq)$ if $f:A\longrightarrow B$ and $g:B\longrightarrow A$ are functions such that for all $a\in A$ and $b\in B$, $a\leq g(b)$ if and only if $b\preccurlyeq f(a)$. It is well known that $(f,g)$ is a Galois connection if and only if $gf,fg$ are inflationary and $f,g$ are antitone \citep[Theorem 2]{gar} .
\begin{proposition}\label{fgpropep}\cite{gar}
Let $(f,g)$ be a Galois connection between posets $\mathscr{A}$ and $\mathscr{B}$. The following assertions hold:
\begin{enumerate}
  \item  [$(1)$ \namedlabel{fgpropep1}{$(1)$}] $fgf=f$ and $gfg=g$;
  \item  [$(2)$ \namedlabel{fgpropep2}{$(2)$}] if $\vee X$ exists for some $X\subseteq A$ then $\wedge f(X)$ exists and $\wedge f(X)=f(\vee X)$;
  \item  [$(3)$ \namedlabel{fgpropep3}{$(3)$}] $gf$ is a closure operator on $\mathscr{A}$ and $\mathscr{C}_{gf}=g(B)$.
\end{enumerate}
\end{proposition}
\begin{proposition}\label{gchkpro}
Let $\mathfrak{A}$ be a residuated lattice and $\Gamma$ be a collection of filters in $\mathfrak{A}$. The pair $(h,k)$ is a Galois connection on $\mathcal{P}(\Gamma)$.
\end{proposition}
\begin{proof}
It follows by Proposition \ref{1hopro}\ref{1hopro1}.
\end{proof}
\begin{corollary}\label{hukerprop}
Let $\mathfrak{A}$ be a residuated lattice and $\Gamma$ be a collection of filters in $\mathfrak{A}$. The following assertions hold for any $X\subseteq A$ and $\mathcal{F}\subseteq \Gamma$:
\begin{enumerate}
  \item  [$(1)$ \namedlabel{hukerprop1}{$(1)$}] $kh$ and $hk$ are inflationary;
  \item  [$(2)$ \namedlabel{hukerprop2}{$(2)$}] $h$ and $k$ are antitone;
  \item  [$(3)$ \namedlabel{hukerprop3}{$(3)$}] $\cap_{X\in \mathscr{X}}h(X)=h(\cup \mathscr{X})$ for any $\mathscr{X}\subseteq \mathcal{P}(A)$;
  \item  [$(4)$ \namedlabel{hukerprop4}{$(4)$}] $\cap_{\mathcal{F}\in \mathscr{F}}k(\mathcal{F})=k(\cup \mathscr{F})$ for any $\mathscr{F}\subseteq \mathcal{P}(\Gamma)$;
  \item  [$(5)$ \namedlabel{hukerprop5}{$(5)$}] $hkh(X)=h(X)$ and $khk(\mathcal{F})=k(\mathcal{F})$;
  \item  [$(6)$ \namedlabel{hukerprop6}{$(6)$}] $hk$ is a closure operator on $\Gamma$ and $\mathscr{C}_{hk}=\{h(X)|X\subseteq A\}$.
\end{enumerate}
\end{corollary}
\begin{proof}
The proof is an immediate consequence of Proposition \ref{fgpropep} and Corollary \ref{gchkpro}.
\end{proof}
\begin{proposition}\label{22hopro}
Let $\mathfrak{A}$ be a residuated lattice and $\Gamma$ be a collection of filters in $\mathfrak{A}$. The following assertions hold for any $x,y\in A$:
\begin{enumerate}
\item  [$(1)$ \namedlabel{22hopro1}{$(1)$}] if $x\leq y$, then $h(x)\subseteq h(y)$;
\item  [$(2)$ \namedlabel{22hopro2}{$(2)$}] $h(x)\cup h(y)\subseteq h(x\vee y)$;
\item  [$(3)$ \namedlabel{22hopro3}{$(3)$}] $h(x)\cap h(y)=h(x\odot y)$.
\end{enumerate}
\end{proposition}
\begin{proof}
\begin{enumerate}
  \item [\ref{22hopro1}:] Let $x\leq y$. By \textsc{Remark} \ref{genfilprop}\ref{genfilprop2} we have $\mathscr{F}(y)\subseteq \mathscr{F}(x)$ and by Proposition \ref{1hopro}\ref{1hopro2} and Corollary \ref{hukerprop}\ref{hukerprop2} follows that $h(x)\subseteq h(y)$.
  \item [\ref{22hopro2}:] By \textsc{Remark} \ref{genfilprop}\ref{genfilprop3} and Proposition \ref{1hopro}(\ref{1hopro2} and \ref{1hopro5}) we obtain that $h(x)\cup h(y)\subseteq h(\mathscr{F}(x)\cap\mathscr{F}(y))=h(\mathscr{F}(x\vee y))=h(x\vee y)$.
  \item [\ref{22hopro3}:] By \textsc{Remark} \ref{genfilprop}\ref{genfilprop4}, Proposition \ref{1hopro}\ref{1hopro2} and Corollary \ref{hukerprop}\ref{hukerprop3} we obtain that $h(x)\cap h(y)=h(\{x,y\})=h(\mathscr{F}(\{x,y\}))=h(\mathscr{F}(x\odot y))=h(x\odot y)$.
\end{enumerate}
\end{proof}

Recalling that a collection $\Pi$ of prime filters of a residuated lattice $\mathfrak{A}$ is called \textit{full} if any proper filter of $\mathfrak{A}$ is contained in some member of $\Pi$. Obviously, $Max(\mathfrak{A})$ and $Spec(\mathfrak{A})$ are full sets of prime filters.
\begin{proposition}\label{3hopro}
Let $\mathfrak{A}$ be a residuated lattice and $\Pi$ be a collection of prime filters in $\mathfrak{A}$. The following assertions hold for any $X,Y\subseteq A$ and $x,y\in A$:
\begin{enumerate}
\item  [$(1)$ \namedlabel{3hopro1}{$(1)$}] $h(X)\cup h(Y)=h(\mathscr{F}(X)\cap\mathscr{F}(Y))$;
\item  [$(2)$ \namedlabel{3hopro2}{$(2)$}] $h(x)\cup h(y)=h(x\vee y)$;
\item  [$(3)$ \namedlabel{3hopro3}{$(3)$}] if $\mathscr{F}(X)=A$, then $h(X)=\emptyset$;
\item  [$(4)$ \namedlabel{3hopro4}{$(4)$}] if $\Pi$ is full, then $h(X)=\emptyset$ if and only if $\mathscr{F}(X)=A$;
\item  [$(5)$ \namedlabel{3hopro5}{$(5)$}] if $\Pi=Spec(\mathfrak{A})$, then $kh(X)=\mathscr{F}(X)$;
\item  [$(6)$ \namedlabel{3hopro6}{$(6)$}] if $\Pi=Spec(\mathfrak{A})$, then $h(X)=h(Y)$ if and only if $\mathscr{F}(X)=\mathscr{F}(Y)$;
\item  [$(7)$ \namedlabel{3hopro7}{$(7)$}] $h(X)\cup h(X^{\perp})=\Gamma$.
\end{enumerate}
\end{proposition}
\begin{proof}
\begin{enumerate}
  \item [\ref{3hopro1}:] Let $P\in h(\mathscr{F}(X)\cap\mathscr{F}(Y))$. Hence, $\mathscr{F}(X)\cap\mathscr{F}(Y)\subseteq P$ and in spirit of Proposition \ref{preqpro} follows that $X\subseteq \mathscr{F}(X)\subseteq P$ or $Y\subseteq \mathscr{F}(Y)\subseteq P$. It states that $P\in h(X)\cup h(Y)$. The converse inclusion follows by Proposition \ref{1hopro}\ref{1hopro5}.
  \item [\ref{3hopro2}:] By \textsc{Remark} \ref{genfilprop}\ref{genfilprop3}, Proposition \ref{1hopro}\ref{1hopro2} and \ref{3hopro1} follows that $h(x)\cup h(y)=h(\mathscr{F}(x)\cap\mathscr{F}(y))=h(\mathscr{F}(x\vee y))=h(x\vee y)$.
  \item [\ref{3hopro3}:] It is evident, since each prime filter is proper.
  \item [\ref{3hopro4}:] Let $\Pi$ be full and $h(X)=\emptyset$. If $\mathscr{F}(X)$ is proper, then there exists $P\in \Pi$ so that $X\subseteq \mathscr{F}(X)\subseteq P$. It means that $P\in h(X)$; a contradiction. The converse follows by \ref{3hopro3}.
  \item [\ref{3hopro5}:] It follows by Corollary \ref{intprimfilt}\ref{intprimfilt2}.
  \item [\ref{3hopro6}:] Let $\Pi=Spec(\mathfrak{A})$. If $h(X)=h(Y)$, applying \ref{3hopro5}, it results that $\mathscr{F}(X)=\mathscr{F}(Y)$. Conversely, if $\mathscr{F}(X)=\mathscr{F}(Y)$, applying Proposition \ref{1hopro}\ref{1hopro2}, it follows that $h(X)=h(Y)$.
  \item [\ref{3hopro7}:] By \ref{3hopro1} and Proposition \ref{1hopro}\ref{1hopro3} we have
  \[h(X)\cup h(X^{\perp})=h(\mathscr{F}(X)\cap X^{\perp})=h(1)=\Gamma.\]
\end{enumerate}
\end{proof}

A closure operator is called \textit{topological} if it preserves finite unions and empty-set. If $cl$ is a topological closure operator on $A$, then $\{X|cl(X^{c})=X^{c}\}$ is a topology on $A$ which is called the generated topology by the topological closure operator $cl$ \citep[Proposition 1.2.7]{eng}
\begin{theorem}\label{clotopo}
Let $\mathfrak{A}$ be a residuated lattice and $\Pi$ be a collection of prime filters in $\mathfrak{A}$. Then $hk$ is a topological closure operator on $\Pi$.
\end{theorem}
\begin{proof}
By Corollary \ref{hukerprop}\ref{hukerprop6} follows that $hk$ is a closure operator on $\Pi$. Also, we have $hk(\emptyset)=h(A)=\emptyset$. By Corollary \ref{hukerprop}\ref{hukerprop4} and Proposition \ref{3hopro}\ref{3hopro1} for any $\mathcal{F},\mathcal{G}\subseteq \Pi$ we have $hk(\mathcal{F}\cup \mathcal{G})=h(k(\mathcal{F})\cup k(\mathcal{G}))=hk(\mathcal{F})\cup hk(\mathcal{G})$.
\end{proof}

Theorem \ref{clotopo} ensure that for any collection of prime filters $\Pi$ in a residuated lattice $\mathfrak{A}$ the topological closure operator $hk$ induced a topology on $\Pi$. This topology is called \textit{the hull-kernel topology} (equivalently, \textit{Zariski topology} or \textit{Jacobson's topology} or \textit{Stone's topology}) and denoted by $\tau_{h}$. Applying Corollary \ref{hukerprop}\ref{hukerprop6}, it follows that any closed set of $(\Pi;\tau_{h})$ is a hull of some subset of $A$. It is easy to see that the collection $\{h(x)|x\in A\}$ is a base for the closed sets.

Applying Proposition \ref{1hopro}\ref{1hopro3} and \ref{22hopro}\ref{22hopro3}, it follows that $\{h(x)|x\in A\}$ is a base for a topological space. The generated topology by this base is called the \textit{dual hull-kernel topology} and denoted by $\tau_{d}$.

If $\Pi$ is a collection of prime filters of a residuated lattice $\mathfrak{A}$, let us denote $\Pi\setminus h(X)$ by $d(X)$ for any $X\subseteq A$. Hence, $d(X)=\{P\in \Pi|X\nsubseteq P\}$. Therefore, the family  $\{d(X)\}_{X\subseteq A}$ is the family of open sets of the space $(\Pi;\tau_{h})$. For any $x\in A$ we denote $d(\{x\})$ by $d(x)$.
\begin{proposition}\label{oprispeprop}
 Let $\mathfrak{A}$ be a residuated lattice and $\Pi$ be a collection of prime filters in $\mathfrak{A}$. The following assertions hold for any $X,Y\subseteq A$:
 \begin{enumerate}
   \item [(1) \namedlabel{oprispeprop1}{(1)}] $X\subseteq Y$ implies $d(X)\subseteq d(Y)$;
   \item [(2) \namedlabel{oprispeprop2}{(2)}] $d(X)=\emptyset$ if and only if $X\subseteq \bigcap \Pi$. In particular, $d(\emptyset)=d(1)=\emptyset$;
   \item [(3) \namedlabel{oprispeprop3}{(3)}] if $\mathscr{F}(X)=A$, then $d(X)=\Pi$. In particular, $d(A)=d(0)=\Pi$;
   \item [(4) \namedlabel{oprispeprop4}{(4)}] if $\Pi$ is full, then $d(X)=\Pi$ if and only if $\mathscr{F}(X)=A$;
   \item [(5) \namedlabel{oprispeprop5}{(5)}] $\cup_{X\in \mathcal{X}}d(X)=d(\cup \mathcal{X})$ for any $\mathcal{X}\subseteq \mathcal{P}(A)$;
   \item [(6) \namedlabel{oprispeprop6}{(6)}] $d(X)=d(\mathscr{F}(X))$;
   \item [(7) \namedlabel{oprispeprop7}{(7)}] $kd(X)=(\bigcap \Pi:X)$.
   \item [(8) \namedlabel{oprispeprop8}{(8)}] if $\bigcap \Pi=\{1\}$, then $kd(X)=X^{\perp}$. In particular, if $\Pi=Spec(\mathfrak{A})$ or $\Pi=Min(\mathfrak{A})$, then $kd(X)=X^{\perp}$.
   \item [(9) \namedlabel{oprispeprop9}{(9)}] $d(X)\cap d(Y)=d(\mathscr{F}(X)\cap \mathscr{F}(Y))$;
   \item [(10) \namedlabel{oprispeprop10}{(10)}] if $\Pi=Spec(\mathfrak{A})$, then $d(X)=d(Y)$ if and only if $\mathscr{F}(X)=\mathscr{F}(Y)$.
 \end{enumerate}
 \end{proposition}
\begin{proof}
\item [\ref{oprispeprop1}:] By duality, it follows by Corollary \ref{hukerprop}\ref{hukerprop2}.
\item [\ref{oprispeprop2}:] By duality, it follows by Proposition \ref{1hopro}\ref{1hopro3}.
\item [\ref{oprispeprop3}:] By duality, it follows by Proposition \ref{1hopro}\ref{1hopro4} and \ref{3hopro}\ref{3hopro3}.
\item [\ref{oprispeprop4}:] By duality, it follows by Proposition \ref{3hopro}\ref{3hopro4}.
\item [\ref{oprispeprop5}:] By duality, it follows by Corollary \ref{hukerprop}\ref{hukerprop3}.
\item [\ref{oprispeprop6}:] By duality, it follows by Proposition \ref{1hopro}\ref{1hopro2}.
\item [\ref{oprispeprop7}:] It follows by Proposition \ref{ccapspecx}.
\item [\ref{oprispeprop8}:] It is evident by \ref{oprispeprop7}.
\item [\ref{oprispeprop9}:] By duality, it follows by Proposition \ref{3hopro}\ref{3hopro1}.
\item [\ref{oprispeprop10}:] By duality, it follows by Proposition \ref{3hopro}\ref{3hopro6}.
\end{proof}

\begin{proposition}\label{1adprispeprop}
Let $\mathfrak{A}$ be a residuated lattice and $\Pi$ be a collection of prime filters in $\mathfrak{A}$. The following assertions hold:
 \begin{enumerate}
   \item [(1) \namedlabel{1adprispeprop1}{(1)}] If $x\leq y$, then $d(y)\subseteq d(x)$;
   \item [(2) \namedlabel{1adprispeprop2}{(2)}] $d(x)\cap d(y)=d(x\vee y)$;
   \item [(3) \namedlabel{1adprispeprop3}{(3)}] $d(x)\cup d(y)=d(x\odot y)$;
   \item [(4) \namedlabel{1adprispeprop4}{(4)}] $h(x)\subseteq d(\neg x)$.
 \end{enumerate}
 \end{proposition}
\begin{proof}
\ref{1adprispeprop1}, \ref{1adprispeprop2} and \ref{1adprispeprop3}, by duality, follows by Proposition \ref{22hopro}.
\item [\ref{1adprispeprop4}:] Let $P\in h(a)$. So $a\in P$. If $\neg a\in P$, then $0=a\odot \neg a\in P$ and it is a contradiction. Hence, $\neg a\notin P$ and so $P\in d(\neg a)$.
\end{proof}

Since any closed set of $(\Pi;\tau_{h})$ is a hull of some subset of $A$ so we obtain that $\tau_{h}=\{d(X)|X\subseteq A\}$.
\begin{proposition}
Let $\mathfrak{A}$ be a residuated lattice and $\Pi$ be a collection of prime filters of $\mathfrak{A}$. The family $\{d(x)\}_{x\in A}$ is a basis for the topology space $(\Pi;\tau_{h})$.
\end{proposition}
\begin{proof}
Let $d(X)$ be an arbitrary open subset of $\Pi$ for some $X\subseteq A$. By Proposition \ref{oprispeprop}\ref{oprispeprop5} follows that $d(X)=d(\cup_{x\in X} x)=\cup_{x\in X} d(x)$.
\end{proof}
The family $\{d(x)\}_{x\in A}$ is called a \textit{base} of the topological space $(\Pi;\tau_{h})$. In the next theorem we observe that any collection of prime filters in a residuated lattice is a $T_0$ space with the hull-kernel and the dual hull-kernel topology.
\begin{theorem}\label{t0spacecon}
Let $\mathfrak{A}$ be a residuated lattice and $\Pi$ be a collection of prime filters in $\mathfrak{A}$. Then $(\Pi;\tau_{h})$ and $(\Pi;\tau_{d})$ are $T_0$ spaces.
\end{theorem}
\begin{proof}
Let $P$ and $Q$ be two distinct element of $\Pi$. So $P\nsubseteq Q$ or $Q\nsubseteq P$. Suppose that $P\nsubseteq Q$. Hence there exists $a\in P\setminus Q$ and so $Q\in d(a)$ and $P\notin d(a)$. It follows that $(\Pi;\tau_{h})$ is a $T_0$ space. Analogously, we can show that $(\Pi;\tau_{d})$ is a $T_0$ space.
\end{proof}

In the next theorem we give a necessary and sufficient condition for a collection of prime filters in a residuated lattice be a $T_1$ space with the hull-kernel and the dual hull-kernel topology.
\begin{theorem}\label{t1spa}
Let $\mathfrak{A}$ be a residuated lattice and $\Pi$ be a collection of prime filters in $\mathfrak{A}$. The following assertions hold:
 \begin{enumerate}
   \item [(1) \namedlabel{t1spa1}{(1)}] $(\Pi;\tau_{d})$ is a $T_1$ space;
   \item [(2) \namedlabel{t1spa2}{(2)}] $(\Pi;\tau_{h})$ is a $T_1$ space;
   \item [(3) \namedlabel{t1spa3}{(3)}] $\Pi$ is an antichain.
 \end{enumerate}
In particular, $Max(\mathfrak{A})$ and $Min(\mathfrak{A})$ are $T_1$ space with both the hull-kernel topology and the dual hull-kernel topology.
\end{theorem}
\begin{proof}
\item \ref{t1spa1}$\Leftrightarrow$\ref{t1spa2}: It is straightforward by this fact that $d(a)=\Pi\setminus h(a)$ for any $a\in A$.
\item \ref{t1spa2}$\Rightarrow$\ref{t1spa3}:  Let $(\Pi;\tau_{h})$ be a $T_1$ space and $P,Q$ be two distinct element of $\Pi$. So there exists $a\in A$ such that $Q\in d(a)$ and $P\notin d(a)$. It follows that $a\in P\setminus Q$ and so $P\nsubseteq Q$. Analogously, we can show that $Q\nsubseteq P$ and it results that $\Pi$ is an antichain.
\item \ref{t1spa3}$\Rightarrow$\ref{t1spa2}: Let $\Pi$ be an antichain and $P,Q\in \Pi$ such that $P\neq Q$. So $P\nsubseteq Q$ and $Q\nsubseteq P$. Hence, there exist $a\in P\setminus Q$ and $b\in Q\setminus P$. It shows that $Q\in d(a)$, $P\notin d(a)$, $P\in d(b)$ and $Q\notin d(b)$. Thus $(\Pi;\tau_{h})$ is a $T_1$ space.
\end{proof}

In the next theorem we give a necessary and sufficient condition for a collection of prime filters in a residuated lattice be a Hausdorff space with the hull-kernel and the dual hull-kernel topology.
\begin{theorem}\label{hauspro}
Let $\mathfrak{A}$ be a residuated lattice and $\Pi$ be a collection of prime filters in $\mathfrak{A}$. The following assertions hold:
 \begin{enumerate}
   \item [(1) \namedlabel{hauspro1}{(1)}] $(\Pi;\tau_{d})$ is a Hausdorff space;
   \item [(2) \namedlabel{hauspro2}{(2)}] $(\Pi;\tau_{h})$ is a Hausdorff space;
   \item [(3) \namedlabel{hauspro3}{(3)}] $\Pi$ is $\bigcap \Pi$-closed;
 \end{enumerate}
\end{theorem}
\begin{proof}
\item \ref{hauspro1}$\Leftrightarrow$\ref{hauspro2}: It is straightforward by this fact that $d(a)=\Pi\setminus h(a)$ for any $a\in A$.
\item \ref{hauspro2}$\Rightarrow$\ref{hauspro3}: Assume that $P_1$ and $P_2$ are two distinct elements of $\Pi$. So there exist two basic open neighbourhoods $d(a_1)$ and $d(a_2)$ of $P_1$ and $P_2$, respectively, such that $d(a_1)\cap d(a_2)=\emptyset$. It states that $a_1\notin P_1$ and $a_2\notin P_2$ and by applying Proposition \ref{oprispeprop}\ref{oprispeprop2} and \ref{1adprispeprop}\ref{1adprispeprop2}, it follows that $a_1\vee a_2\in \bigcap \Pi$.
\item \ref{hauspro3}$\Rightarrow$\ref{hauspro2}: Let $P_1$ and $P_2$ be two distinct elements of $\Pi$. So there exist $a_1\notin P_1$ and $a_2\notin P_2$ such that $a_1\vee a_2\in \bigcap \Pi$. therefore, $d(a_1)$ and $d(a_2)$ are two basic open neighbourhoods of $P_1$ and $P_2$, respectively such that $d(a_1)\cap d(a_2)=\emptyset$.
\end{proof}
\begin{corollary}\label{hausantichmtl}
Let $\mathfrak{A}$ be a MTL algebra and $\Pi$ be a collection of prime filters of $\mathfrak{A}$. $(\Pi;\tau_{h})$ is a hausdorff space if and only if $\Pi$ is an antichain.
\end{corollary}
\begin{proof}
It is an immediate consequence of Corollary \ref{mtlcloant} and Theorem \ref{hauspro}.
\end{proof}

\begin{lemma}\label{retractlemma}
Let $\mathfrak{A}$ be a residuated lattice, $\Pi$ be a collection of prime filters of $\mathfrak{A}$ and $P,Q\in \Pi$. The following assertions are equivalent:
\begin{enumerate}
   \item [(1) \namedlabel{retractlemma1}{(1)}] $P\subseteq Q$;
   \item [(2) \namedlabel{retractlemma2}{(2)}] $Q\in cl_{\tau_{h}}(\{P\})$;
   \item [(3) \namedlabel{retractlemma3}{(3)}] $P\in cl_{\tau_{d}}(\{Q\})$.
 \end{enumerate}
\end{lemma}
\begin{proof}
\item \ref{retractlemma1}$\Rightarrow$\ref{retractlemma2}: If $d(a)$ is a basic neighborhood of $Q$, then $a\notin Q$ and it implies that $a\notin P$. It shows that $P\in d(a)$ and hence $Q\in cl_{\tau_{h}}(\{P\})$.
\item \ref{retractlemma2}$\Rightarrow$\ref{retractlemma3}: If $h(a)$ is a basic neighborhood of $P$, then $a\in P$. Let $a\notin Q$. Thus $Q\in d(a)$ and it implies that $P\in d(a)$; a contradiction.
\item \ref{retractlemma2}$\Rightarrow$\ref{retractlemma3}: Let $a\in P$. So $P\in h(a)$ and it implies that $Q\in h(a)$.
\end{proof}

Recalling that a \textit{retraction} is a continuous mapping from a topological space into a subspace which preserves the position of all points in that subspace.
\begin{proposition}\label{retrano}
Let $\mathfrak{A}$ be a residuated lattice and $\Pi$ be an antichain of prime filters of $\mathfrak{A}$. The following assertions hold:
\begin{enumerate}
   \item [(1) \namedlabel{retrano1}{(1)}] If $(\Pi;\tau_{h})$ is retract of $(S_{\Pi};\tau_{h})$, then $(\Pi;\tau_{h})$ is a Hausdorff space;
   \item [(2) \namedlabel{retrano2}{(2)}] if $(\Pi;\tau_{d})$ is retract of $(S_{\Pi};\tau_{d})$, then each element of $S_{\Pi}$ is contained in a unique element of $\Pi$.
 \end{enumerate}
\end{proposition}
\begin{proof}
\item [\ref{retrano1}:] Let $f:(S_{\Pi};\tau_{h})\longrightarrow (\Pi;\tau_{h})$ be a retraction and $P\in \Pi$. Let us we denote the closed set $f^{-1}(P)$ by $F_{P}$. We claim that $P$ is the unique element of $\Pi$ containing $\bigcap F_{P}$. Suppose that $\bigcap F_{P}\subseteq P'$ for some $P'\in \Pi$. If $d(a)\in \mathscr{N}_{P'}$, then $a\notin P'$ and it implies that $a\notin \bigcap F_{P}$. Hence, there exits $Q\in F_{P}$ such that $a\notin Q$. So $Q\in d(a)\cap F_{P}$ and it shows that $P'\in \overline{F_{P}}=F_{P}$. Therefore, $P'\subseteq P$ and it implies that $P'=P$. By Corollary \ref{dfppr} follows that $D_{\bigcap \Pi}(P)=F_{P}$ and it means that $P$ is the unique element of $\Pi$ containing $D_{\bigcap \Pi}(P)$. Hence, $\Pi$ is a Hausdorff space due to Theorem \ref{hauspro}.

\item [\ref{retrano2}:] Let $f:(S_{\Pi};\tau_{d})\longrightarrow (\Pi;\tau_{d})$ be a retraction and $Q\in S_{\Pi}$. Assume that $Q\subseteq P$ for some $P\in \Pi$. By Lemma \ref{retractlemma}, we have $Q\in \overline{\{P\}}$ and by Proposition \ref{comcont}\ref{comcont6} we obtain that
    \[f(Q)\in f(\overline{\{P\}})\subseteq \overline{f(\{P\})}=\overline{\{P\}}.\]
    By Lemma \ref{retractlemma} we have $f(Q)\subseteq P$. Since $\Pi$ is an antichain so $f(Q)=P$. It holds the result.
\end{proof}
\begin{remark}\label{remarkretr}
By the above proposition if $\Pi$ is an antichain of prime filters of a residuated lattice $\mathfrak{A}$, then there is a unique retraction $f:(S_{\Pi};\tau_{h})\longrightarrow (\Pi;\tau_{h})$, defined by $f(Q)=P$ where $P$ is the unique element of $\Pi$ containing $Q$ for any $Q\in S_{\Pi}$.
\end{remark}

\begin{theorem}\label{hausnorm}
Let $\mathfrak{A}$ be a residuated lattice. The following assertions are equivalent:
\begin{enumerate}
   \item [(1) \namedlabel{hausnorm1}{(1)}] Each prime filter of $\mathfrak{A}$ is contained in a unique maximal filter;
   \item [(2) \namedlabel{hausnorm2}{(2)}] $(Max(\mathfrak{A});\tau_{h})$ is retract of $(Spec(\mathfrak{A});\tau_{h})$;
   \item [(3) \namedlabel{hausnorm3}{(3)}] $(Max(\mathfrak{A});\tau_{h})$ is a Hausdorff space.
 \end{enumerate}
\end{theorem}
\begin{proof}
\item \ref{hausnorm1}$\Rightarrow$\ref{hausnorm2}: For any $P\in Spec(\mathfrak{A})$ suppose that $M_{P}$ is the unique maximal filter containing $P$. Define $f:Spec(\mathfrak{A})\longrightarrow Max(\mathfrak{A})$ by $f(P)=M_{P}$. Consider a basic closed set $\mathcal{H}=h_{Max(\mathfrak{A})}(a)$ for some $a\in A$. We claim that $\mathcal{F}=f^{-1}(\mathcal{H})$ is a closed set in $Spec(\mathfrak{A})$. Let $F=k(\mathcal{F})$ and $C=\bigcup \mathcal{H}$. Let $P\in \overline{\mathcal{F}}=hk(\mathcal{F})=h(F)$ and so $F\subseteq P$. We have $F\subseteq C\cap P$ and it means that $F\cap (C\cap P)^{c}=\emptyset$. Since $C^{c}$ and $P^{c}$ are $\vee$-closed subsets of $\mathfrak{A}$ so $\mathscr{C}(C^{c}\cup P^c)=\{x\vee y|x\notin C,~y\notin P\}$. Let us we denote $\mathscr{C}(C^{c}\cup P^c)$ by $\mathscr{C}$. Therefore, we have $(C\cap P)^{c}=C^{c}\cup P^c\subseteq \mathscr{C}$. Let $x\vee y\in F$ for some $x\in C^c$ and $y\in P^c$. Since $y\notin P$ so $y\notin F$. It follows that there exists $Q\in \mathcal{F}$ such that $y\notin Q$. On the other hand, $x\vee y\in F\subseteq Q$, which it implies that $x\in Q\subseteq C$; a contradiction. So $F\cap \mathscr{C}=\emptyset$. Applying Theorem \ref{prfilth}, it follows that there exists a prime filter $Q$ such that $Q\cap \mathscr{C}=\emptyset$ and $F\subseteq Q$. It results that $Q\subseteq \mathscr{C}^{c}\subseteq C\cap P$. Let $\mathscr{F}(Q,a)=A$. So $q\odot a^{n}=0$ for some $q\in Q$ and integer $n$. Since $Q\subseteq C$ so there exists some $M\in h(a)$ such that $q\in M$, but it implies that $0\in M$; a contradiction. Thus, $\mathscr{F}(Q,a)\in M$ for some $M\in \mathcal{H}$. Hence, $P\subseteq Q\subseteq \mathscr{F}(Q,a)\subseteq M$, it follows that $P\in \mathcal{F}$. It states that $\mathcal{F}$ is a closed set in $(Spec(\mathfrak{A});\tau_{h})$ and so $f$ is a continuous function. Also, it is obvious that $f(M)=M$ for any $M\in Max(\mathfrak{A})$. It shows that $f$ is retract.
\item \ref{hausnorm2}$\Rightarrow$\ref{hausnorm3}: It is evident by Proposition \ref{retrano}\ref{retrano1}.
\item \ref{hausnorm3}$\Rightarrow$\ref{hausnorm1}: Applying Proposition \ref{chauspro} and \ref{hauspro}, it follows that each prime filter is contained in a unique maximal filter.

\end{proof}
%
\begin{lemma}\label{hasnorm}
Let $\mathfrak{A}$ be a residuated lattice, $\Pi$ be a collection of prime filters in $\mathfrak{A}$ and $f:(S_{\Pi};\tau_{h})\longrightarrow (\Pi;\tau_{h})$ be a retraction. If $\Pi$ is a $T_4$ space and $S_{\Pi}$ is a compact space, then $S_{\Pi}$ is a normal space.
\end{lemma}
\begin{proof}
Since $(\Pi;\tau_{h})$ is a $T_4$ space so $\Pi$ is an antichain and by \textsc{Remark} \ref{remarkretr} follows that $f(Q)=P$ where $P$ is the unique element of $\Pi$ containing $Q$ for any $Q\in S_{\Pi}$. By Lemma \ref{norhau}, $f$ is a closed map. Let $C_1$ and $C_2$ be two disjoint closed sets in $S_{\Pi}$, so $f(C_1)$ and $f(C_2)$ are disjoint closed set in $\Pi$. Since $\Pi$ is normal, there exist disjoint open neighbourhoods $N_1$ and $N_2$ of $f(C_1)$  and $f(C_2)$ in $\Pi$, respectively. One can see that  $f^{-1}(N_1)$ and $f^{-1}(N_2)$ are disjoint open neighbourhoods of $C_1$ and $C_2$, respectively.
\end{proof}
\begin{lemma}\label{haulemm}
Let $\mathfrak{A}$ be a residuated lattice and $\Pi$ be an antichain of prime filters in $\mathfrak{A}$. If $(S_{\Pi};\tau_{h})$ is a normal space, then $(\Pi;\tau_{h})$ is a Hausdorff space.
\end{lemma}
\begin{proof}
Let $P\in \Pi$ and $Q\in cl_{S_{\Pi}}(\{P\})$. By Lemma \ref{retractlemma} we have $P\subseteq Q$ and it implies that $p=Q$ and so $cl_{S_{\Pi}}(\{P\})=\{P\}$. It shows that $\{P\}$ is a closed set in $S_{\Pi}$. Let $P_1$ and $P_2$ be distinct elements of $\Pi$. Since $S_{\Pi}$ is normal so there exist disjoint neighborhoods $N_1$ and $N_2$ for $P_1$ and $P_2$ in $S_{\Pi}$, respectively. Therefore, $N_1\cap \Pi$ and $N_2\cap \Pi$ are disjoint neighborhoods for $P_1$ and $P_2$ in $\Pi$, respectively.
\end{proof}
\begin{proposition}\label{compapropo}
Let $\mathfrak{A}$ be a residuated lattice and $\Pi$ be a collection of prime filters in $\mathfrak{A}$. The following assertions hold:
\begin{enumerate}
   \item [(1) \namedlabel{compapropo1}{(1)}] $(\Pi;\tau_{h})$ is a compact space, provided that $\Pi$ is full. In particular, $(Spec(\mathfrak{A});\tau_{h})$ and $(Max(\mathfrak{A});\tau_{h})$ are compact;
   \item [(2) \namedlabel{compapropo2}{(2)}] $(\Pi;\tau_{d})$ is a compact space, provided that $\Pi$ contains $Min_{\bigcap \Pi}(\mathfrak{A})$. In particular, $(Spec(\mathfrak{A});\tau_{d})$ is compact;
 \end{enumerate}
\end{proposition}
\begin{proof}
\item [\ref{compapropo1}:] Let $\Pi=\bigcup_{x\in X}d(x)$. By Proposition \ref{oprispeprop}\ref{oprispeprop5} follows that $\Pi=d(X)$ and so by Proposition \ref{oprispeprop}\ref{oprispeprop4} we get that $\mathscr{F}(X)=A$. So by \textsc{Remark} \ref{genfilprop}\ref{genfilprop6}, $A=\mathscr{F}(Y)$ for a finite subset $Y\subseteq X$. Therefore, $\Pi=d(Y)=\bigcup_{y\in Y}d(y)$. It holds the result. The remind is evident, since $Spec(\mathfrak{A})$ and $Max(\mathfrak{A})$ are full.
\item [\ref{compapropo2}:] Let $X$ be a subset of $A$ such that for any finite subset $Y\subseteq X$ we have $\bigcap_{y\in Y}d(y)\neq \emptyset$. Using Proposition \ref{1adprispeprop}\ref{1adprispeprop2}, it results that $d(\bigvee Y)\neq\emptyset$ and so by Proposition \ref{oprispeprop}\ref{oprispeprop2} we obtain that $\bigvee Y\notin \bigcap \Pi$. It shows that $\mathscr{C}(X)\cap \bigcap \Pi=\emptyset$. By Lemma \ref{0prfiltth}, there exists a maximal $\vee$-closed subset of $\mathfrak{A}$, named $C_{X}$, such that not meeting $\bigcap \Pi$. By Minimal prime filter theorem, it follows that $\mathscr{M}=A\setminus C_{X}$ is a $\bigcap \Pi$-minimal prime filter and so $\mathscr{M}\in \Pi$. Thus $\mathscr{M}\in d(x)$ for any $x\in X$. It concludes that $\bigcap_{x\in X}d(x)\neq\emptyset$. Hence the result holds. The remind is evident.
\end{proof}
\begin{corollary}\label{hausnorm}
Let $\mathfrak{A}$ be a residuated lattice. Then $(Max(\mathfrak{A});\tau_{h})$ is a Hausdorff space if and only if $(Spec(\mathfrak{A});\tau_{h})$ is a normal space.
\end{corollary}
\begin{proof}
If $Max(\mathfrak{A})$ is a Hausdorff space, then $Max(\mathfrak{A})$ is retract of $Spec(\mathfrak{A})$ by Theorem \ref{hausnorm}, and $Max(\mathfrak{A})$ is a $T_4$ space by Proposition \ref{compapropo}. It states that $Spec(\mathfrak{A})$ is a normal space by Lemma \ref{hasnorm}. The converse is evident by Lemma \ref{haulemm}.
\end{proof}
\begin{corollary}
Let $\mathfrak{A}$ be a MTL algebra. Then, $(Max(\mathfrak{A});\tau_{h})$ is a Hausdorff space and $Spec(\mathfrak{A})$ is a normal space.
\end{corollary}
\begin{proof}
It is a direct result of Corollary \ref{hausantichmtl} and \ref{hausnorm}.
\end{proof}

\section{The space of minimal prime filters}\label{sec4}

In this section, we focused on the space of minimal prime filters in a residuated lattice. From now on, all hulls and kernels refer to $Min(\mathfrak{A})$.

\begin{proposition}\label{hausantich}
Let $\mathfrak{A}$ be a residuated lattice and $F$ be a filter of $\mathfrak{A}$. Then $Min_{F}(\mathfrak{A})$ is a Hausdorff space.
\end{proposition}
\begin{proof}
It follows by Theorem \ref{hauspro} since $\bigcap Min_{F}(\mathfrak{A})=F$ and for any $\mathfrak{m}\in Min_{F}(\mathfrak{A})$ we have $\mathfrak{m}=D_{F}(\mathfrak{m})$.
\end{proof}
\begin{theorem}\label{dmincompa}
Let $\mathfrak{A}$ be a residuated lattice. Then $(Min(\mathfrak{A});\tau_{d})$ is compact.
\end{theorem}
\begin{proof}
It is a direct consequence of Proposition \ref{compapropo}\ref{compapropo2}.
\end{proof}
\begin{proposition}\label{minpro}
Let $\mathfrak{A}$ be a residuated lattice.
\begin{enumerate}
   \item [(1) \namedlabel{minpro1}{(1)}] $kd(X)=X^{\perp}$;
   \item [(2) \namedlabel{minpro2}{(2)}] $h(x)\cap h(x^{\perp})=\emptyset$;
   \item [(3) \namedlabel{minpro3}{(3)}] $d(x)=h(x^{\perp})$ and $d(x^{\perp})=h(x)$. In particular, $hkd(x)=d(x)$;
   \item [(4) \namedlabel{minpro4}{(4)}] $kh(X^{\perp})=X^{\perp}$;
   \item [(5) \namedlabel{minpro5}{(5)}] $h(x)=h(x^{\perp\perp})$;
   \item [(6) \namedlabel{minpro6}{(6)}] $h(x^{\perp})=h(y)$ if and only if $x^{\perp\perp}=y^{\perp}$.
 \end{enumerate}
\end{proposition}
\begin{proof}
\begin{enumerate}
  \item [\ref{minpro1}:] It follows by Proposition \ref{oprispeprop}\ref{oprispeprop8}.
  \item [\ref{minpro2}:] It follows by Theorem \ref{mincor}\ref{mincor3}.
  \item [\ref{minpro3}:] By Proposition \ref{3hopro}\ref{3hopro7} we have $h(x)\cup h(x^{\perp})=Min(\mathfrak{A})$ and by \ref{minpro2} we have $h(x)\cap h(x^{\perp})=\emptyset$. It shows the result.
  \item [\ref{minpro4}:] It is obvious that $X^{\perp}\subseteq kh(X^{\perp})$. Conversely, let $a\notin X^{\perp}$. So there exists $x\in X$  so that $a\vee x\neq 1$. It implies that $a\vee x\notin X^{\perp}$. By Theorem \ref{prfilth} and Corollary \ref{primeminimal} there exists a $X^{\perp}$-minimal prime filter $\mathfrak{m}$ which not containing $a\vee x$. Let $m\in \mathfrak{m}=D_{X^{\perp}}(\mathfrak{m})$. So there exists $b\notin \mathfrak{m}$ such that $m\vee b\in X^{\perp}$. It means that $m\vee(b\vee x)=1$, but $b\vee x\notin \mathfrak{m}$ since $b,x\notin \mathfrak{m}$. So $m\in D(\mathfrak{m})$ and it shows that $\mathfrak{m}$ is a minimal prime filter. Consequently, $a\notin kh(X^{\perp})$. So the result holds.
  \item [\ref{minpro5}:] By Proposition \ref{hukerprop}\ref{hukerprop5}, \ref{minpro1} and \ref{minpro3} we have $h(x^{\perp\perp})=hkd(x^{\perp})=hkh(x)=h(x)$.
  \item [\ref{minpro6}:] It is evident by \ref{minpro1} and \ref{minpro5}.
\end{enumerate}
\end{proof}

Recalling that a topological space $(A;\tau)$ is called \textit{zero-dimensional} if it has a base for open sets consisting of clopen sets. Also, $(A;\tau)$ is called \textit{totally disconnected} if for any distinct points $a,b\in A$, there exists a
clopen subset $U$ such that $x\in U$ and $y\notin U$. It is well-known that any $T_1$ zero-dimensional space is totally disconnected \citep[Theorem 6.2.1]{eng}. Let $\tau$ and $\zeta$ be two topologies on $A$. We say that $\tau$ is finer than $\zeta$, if $\zeta\subseteq \tau$.
\begin{corollary}\label{minspapro}
Let $\mathfrak{A}$ be a residuated lattice. The following assertions hold:
 \begin{enumerate}
   \item [(1) \namedlabel{minspapro1}{(1)}] $(Min(\mathfrak{A});\tau_{h})$ is zero-dimensional and consequently totally disconnected;
   \item [(2) \namedlabel{minspapro2}{(2)}] $\tau_{d}$ is finer than $\tau_{h}$ on $Min(\mathfrak{A})$.
 \end{enumerate}
 \end{corollary}
 \begin{proof}
 \item [\ref{minspapro1}:] Applying Proposition \ref{minpro}\ref{minpro3}, it follows that $d(x)$ is a clopen set for any $x\in A$.
 \item [\ref{minspapro2}:] Applying Proposition \ref{minpro}\ref{minpro3}, it follows that $\{d(x)\}_{x\in A}\subseteq \tau_{d}$ and so $\tau_{h}\subseteq \tau_{d}$.
 \end{proof}

 The notion of $\star$-lattices is introduced by \cite{spe0} as a generalization of distributive pseudo-complemented lattices. This class of distributive lattices are studied extensively by \cite{spe0,spe}.
\begin{definition}\label{qucomresdef}
A residuated lattice $\mathfrak{A}$ is called a $\star$-\textit{residuated lattice} if for any $x\in A$ there exists $y\in A$ such that $x^{\perp\perp}=y^{\perp}$.
\end{definition}

Recalling that for a given topological space $(A;\tau)$ and a subset $X$ of $A$, the subspace topology on $X$ is defined by $\tau _{X}=\{ X\cap U| U\in \tau \}$. It is well-known that a subset of $X$ is open (close) in the subspace topology if and only if it is the intersection of $X$ with an open (close) set in $A$.

\begin{theorem}\label{compmin}
Let $\mathfrak{A}$ be a residuated lattice. The following assertions are equivalent:
\begin{enumerate}
   \item [(1) \namedlabel{compmin1}{(1)}] $\mathfrak{A}$ is a $\star$-residuated lattice;
   \item [(2) \namedlabel{compmin2}{(2)}] $\tau_{h}$ and $\tau_{d}$ coincide on $Min(\mathfrak{A})$;
   \item [(3) \namedlabel{compmin3}{(3)}] $(Min(\mathfrak{A});\tau_{h})$ is compact.
 \end{enumerate}
\end{theorem}
\begin{proof}
\item \ref{compmin1}$\Rightarrow$\ref{compmin2}: Let $x\in A$. By Proposition \ref{minpro}(\ref{minpro3} and \ref{minpro5}) and hypothesis, it follows that $h(x)=h(x^{\perp\perp})=h(y^{\perp})=d(y)$ for some $y\in A$. It shows that $\tau_{h}$ is finer than $\tau_{d}$. The converse follows by Corollary \ref{minspapro}\ref{minspapro2}.
\item \ref{compmin2}$\Rightarrow$\ref{compmin3}: It is straightforward by Proposition \ref{compapropo}\ref{compapropo2}.
\item \ref{compmin3}$\Rightarrow$\ref{compmin1}: Let $x\in A$. Then $h(x)$ is a closed subset of $(Min(\mathfrak{A});\tau_{h})$ and so it is compact in the subspace topology. By Proposition \ref{minpro}\ref{minpro2}, we have $h(x)\cap h(x^{\perp})=\emptyset$. By Proposition \ref{hukerprop}\ref{hukerprop3} we have
    \[\emptyset=h(x)\cap (\bigcap_{t\in x^{\perp}} h(t))=\bigcap_{t\in x^{\perp}} (h(x)\cap h(t))).\]
    Since for any $t\in x^{\perp}$, $h(x)\cap h(t)$ is a closed subset of $h(x)$ in the subspace topology so for some $t_1,\cdots,t_n\in x^{\perp}$ we have
    \[\emptyset=\bigcap_{i=1}^{n} (h(x)\cap h(t_i)))=h(x)\cap (\bigcap_{i=1}^{n} h(t_i))=h(x)\cap h(\odot_{i=1}^{n} t_i).\]
    Take $y=\odot_{i=1}^{n} t_i$. Since $h(x)\cap h(y)=\emptyset$ so $d(x)\cup d(y)=Min(\mathfrak{A})$. On the other hand, $y\in x^{\perp}$ and it shows that $x\vee y=1$. By Proposition \ref{1adprispeprop}\ref{1adprispeprop2}, it follows that $d(x)\cap d(y)=\emptyset$ and so $d(x)=h(y)$. Using Proposition \ref{minpro}\ref{minpro3}, $h(x^{\perp})=h(y)$ and by Proposition \ref{minpro}\ref{minpro6} follows the result.
\end{proof}

\begin{lemma}\label{cloudis}
Let $\mathfrak{A}$ be a residuated lattice and $X\subseteq A$. If $S=\bigcup_{x\in X} h(x^{\perp})$, then the closure of $S$ is $h(X^{\perp})$.
\end{lemma}
\begin{proof}
By Proposition \ref{minpro}\ref{minpro4} follows that $k(S)=\bigcap_{x\in X} kh(x^{\perp})=\bigcap_{x\in X} x^{\perp}=X^{\perp}$ and so $\overline{S}=hk(S)=h(X^{\perp})$.
\end{proof}

Recalling that a topological space is termed \textit{extremally disconnected} if the closure of every open set in it is open \citep[p. 368]{eng}. An extremally disconnected space that is also compact and Hausdorff is sometimes called a \textit{Stonean space} \citep{stra}.
\begin{proposition}\label{propomohem}
Let $\mathfrak{A}$ be a residuated lattice. $(Min(\mathfrak{A});\tau_{h})$ is an extremally disconnected space if and only if $h(X^{\perp})$ is open for any $X\subseteq A$.
\end{proposition}
\begin{proof}
It is an immediate consequence of Lemma \ref{cloudis}.
\end{proof}
\begin{definition}\label{qucomresdef}
A residuated lattice $\mathfrak{A}$ is called a \textit{(countable) $\bigstar$-residuated lattice} if for any (countable) subset $X$ of $A$ there exists $y\in A$ such that $X^{\perp}=y^{\perp}$.
\end{definition}

It is obvious that each $\bigstar$-residuated lattice is a $\star$-residuated lattice, but not the converse.
\begin{proposition}
Let $\mathfrak{A}$ be a residuated lattice. $(Min(\mathfrak{A});\tau_{h})$ is a Stonean space if and only if $\mathfrak{A}$ is a $\bigstar$-residuated lattice.
\end{proposition}
\begin{proof}
Assume that $(Min(\mathfrak{A});\tau_{h})$ is a Stonean space and $X\subseteq A$. By Lemma \ref{cloudis} and Proposition \ref{propomohem}, $h(X^{\perp})$ is clopen and so $d(X^{\perp})$ is clopen. By Theorem \ref{compmin}, $\{h(x)\}_{x\in A}$ is a base for the topological space $(Min(\mathfrak{A});\tau_{h})$; hence, the open set $d(X^{\perp})$ is a union of such basic open set. Since $d(X^{\perp})$ is a closed set of the compact topological space $(Min(\mathfrak{A});\tau_{h})$ so it is also compact; thus, it is a union of a finite number of such sets. So there exist $x_1,\cdots,x_n\in A$ such that $d(X^{\perp})=\bigcup_{i=1}^{n}h(x_i)=h(\odot_{i=1}^{n}x_i)$. Take, $y=\odot_{i=1}^{n}x_i$; therefore, $d(X^{\perp})=h(y)=d(y^{\perp})$. Thus $h(X^{\perp})=h(y^{\perp})$ and by using Proposition \ref{minpro}\ref{minpro4} we get that $X^{\perp}=y^{\perp}$. The converse is obvious by Theorem \ref{compmin}, Lemma \ref{cloudis} and Proposition \ref{propomohem}.
\end{proof}

Recalling that a maximal filter of the poset $\mathcal{P}(\mathds{N})$ is called an ultrafilter of $\mathds{N}$.
\begin{lemma}\label{lemminperp}
Let $\{\mathfrak{m}_{n}\}$ be a sequence of minimal prime filters of a countable $\bigstar$-residuated lattice, $U$ be an ultrafilter on $\mathds{N}$ and $E(x)=\{n\in \mathds{N}|x\in \mathfrak{m}_{n}\}$. Then the set $\mathfrak{m}_{U}=\{x\in A|E(x)\in U\}$ is a minimal prime filter of $\mathfrak{A}$.
\end{lemma}
\begin{proof}
Since $E(x\vee y)=E(x)\cup E(y)$ and $E(x\odot y)\supseteq E(x)\cap E(y)$ so $\mathfrak{m}_{U}$ is a prime filter of $\mathfrak{A}$. Let $a\in \mathfrak{m}_{U}$. So $E(a)\in U$. Thus there exits $x_n\notin \mathfrak{m}_{U}$ such that $a\vee x_{n}=1$ for any $n\in E(a)$. Take $X=\{x_n\}_{n\in E(a)}$; therefore there exists $y\in A$ so that $X^{\perp}=y^{\perp}$. Since $a\in x^{\perp}_{n}$ for any $n\in E(a)$ so $a\in X^{\perp}=y^{\perp}$. It implies that $a\vee y=1$. On the other hand, $y^{\perp}\subseteq x^{\perp}_{n}\subseteq \mathfrak{m}_{n}$ for any $n\in E(a)$, and so $y\notin \mathfrak{m}_{n}$. It follows that $E(y)\cap E(a)=\emptyset$ and it states that $E(y)\notin U$. Consequently, $y\notin \mathfrak{m}_{U}$ and it shows that $a\in D(\mathfrak{m}_{U})$. It holds the result.
\end{proof}

Recalling that a topological space $(A;\tau)$ is called \textit{countably compact} if any countable open cover of $A$ has a finite subcover \citep[p. 202]{eng}. It is well-known that a topological space $(A;\tau)$ is countably compact if and only if every sequence in $A$ has a cluster point \citep[Theorem 3.10.3]{eng}.
\begin{theorem}
Let $\mathfrak{A}$ be a countable $\bigstar$-residuated lattice. then $(Min(\mathfrak{A});\tau_{h})$ is a countably compact space.
\end{theorem}
\begin{proof}
Let $\{\mathfrak{m}_{n}\}$ be a sequence of minimal prime filters of $\mathfrak{A}$ and $U$ be a free ultrafilter on $\mathds{N}$. We show that $\mathfrak{m}_{U}$ is the cluster point of the sequence $\{\mathfrak{m}_{n}\}$. Since $U$ is a free ultrafilter so for any $m\in \mathds{N}$, $\{n\in \mathds{N}\}_{n\geq m}$. Moreover, $x\in k(\{\mathfrak{m}_{n}\}_{n\geq m})=\bigcap_{n\geq m} \mathfrak{m}_{n}$ implies that $\{n\in \mathds{N}\}_{n\geq m}\subseteq E(x)$ and so $E(x)\in U$. Therefore, $k(\{\mathfrak{m}_{n}\}_{n\geq m})\subseteq \mathfrak{m}_{U}$. Consequently, $\mathfrak{m}_{U}\in h(\mathfrak{m}_{U})\subseteq h(k(\{\mathfrak{m}_{n}\}_{n\geq m}))=\overline{\{\mathfrak{m}_{n}\}_{n\geq m}}$. Since it is true for any integer $m$ so $\mathfrak{m}_{U}$ is a cluster point of $\{\mathfrak{m}_{n}\}$.
\end{proof}
\section*{Notes on contributors}

\begin{itemize}
  \item Saeed Rasouli is an assistant professor in the Department of Mathematics at Persian Gulf University, Bushehr, Iran. He obtained his Ph. D from Yazd University, School of Sciences, Department of Mathematics in 2010. His research interests are logical algebras focused on residuated lattices, and hyperstructures.  He has published more than $25$ research papers, especially on algebraic hyperstructures and residuated lattices, in National and International reputed journals.  He has $160$ citations (Google Scholar) with H-index $6$ (Google Scholar).
  \item Amin Dehghani  received the M. Sc degree in the pure mathematics from Iran University of Science and Technology, Tehran, Iran, in 2011. He has received his Ph. D degree in the area of topology from Persian Gulf University, Iran, in 2018.
\end{itemize}



\begin{thebibliography}{}
\bibitem[Atiah and Macdonald(1969)]{ati}
Atiyah, M. F., and I. G. Macdonald. 1969. \emph{Introduction to Commutative Algebra}, Reading, Mass.-Menlo Park, Calif.- London-Don Mills , Ont.: Addison- Wesley Publishing Company.
\bibitem[Birkenmeier, Kim, and Park(1998)]{bir}
Birkenmeier, G. F, J. Y. Kim, and J. K. Park. 1998. ``A characterization of minimal
prime ideals."  \emph{Glasgow Math. J.} 40: 223--236.

\bibitem[Di Nola, Georgescu, and Iorgulescu(2002)]{din0}
Di Nola, A., G. Georgescu, and A. Iorgulescu. 2002. ``Pseudo-BL algebras: Part I." \emph{Multiple Valued Logic} 8: 673--714.
\bibitem[Engelking(1989)]{eng}
Engelking, R., 1989. \emph{General topology}, Sigma series in pure mathematics, vol. 6.
\bibitem[Galatos et al.(2007)]{gal}
Galatos, N., P. Jipsen, T. Kowalski, and H. Ono. 2007. \emph{Residuated lattices: an algebraic glimpse at substructural logics.} Elsevier.
\bibitem[Garc{\'\i}a-Pardo et al.(2013)]{gar}
Garc{\'\i}a-Pardo, F., I. P. Cabrera, P. Cordero, Pablo, and Ojeda-Aciego. 2013. \emph{On Galois connections and soft computing.} International Work-Conference on Artificial Neural Networks Elsevier (224-235). Springer, Berlin, Heidelberg.
\bibitem[Gierz et al.(2003)]{gie}
Gierz, G., Hofmann, K.H., Keimel, K., Lawson, J.D., Mislove, M. and Scott, D.S., 2003. \emph{Continuous lattices and domains} (Vol. 93). Cambridge university press.
%
\bibitem[Henriksen and Jerison(1965)]{hen}
Henriksen, M., and M. Jerison. 1965. ``The space of minimal prime ideals of a commutative ring."  \emph{Transactions of the American Mathematical Society} 115: 110--130.
\bibitem[Idziak(1984)]{idz}
Idziak, P. M. 1984. ``Lattice operations in BCK-algebras."  \emph{Mathematica Japonica} 29: 839--846.
\bibitem[Jayaram(1986)]{jay}
Jayaram, C. 1986. ``Prime $\alpha$-ideals in a $0$-distributive lattice." \emph{Indian J. Pure Appl. Math.} 17(3): 331--337.
\bibitem[Jipsen and Tsinakis(2002)]{jip}
Jipsen, P., C. Tsinakis. 2002. ``A survey of residuated lattices."  \emph{Ordered Algebraic Structures} 7: 19--56.
\bibitem[Kist(1963)]{kis}
Kist, J., 1963. \emph{Minimal prime ideals in commutative semigroups.} Proceedings of the London Mathematical Society 3(1): 31--50.
\bibitem[Leu\c{s}tean(2003)]{leus0}
Leu\c{s}tean, L. 2003. ``The prime and maximal spectra and the reticulation of BL-algebras." \emph{Central Eroupean journal of mathematics} 3: 382--397.

\bibitem[Mundlik, Joshi, and Hala\v{s}(2017)]{mund}
Mundlik, N., V. Joshi, and R. Hala\v{s}. 2017. ``The hull-kernel topology on prime ideals in posets." \emph{Soft Computing} 21(7): 1653--1665.
\bibitem[Pawar, and Thakare(1977)]{paw}
Pawar, Y. S., and N. K. Thakare. 1977. ``pm-Lattices." \emph{algebra universalis} 7(1): 259--263.
\bibitem[Pawar(1978)]{paw1}
Pawar, Y. S. 1978. ``\emph{A study in lattice theory.}" Ph. D. Thesis (submitted to
Shivaji University, Kolhapur).
\bibitem[Speed(1969a)]{spe0}
Speed, T. P. 1969. ``Two congruences on distributive lattices." \emph{Bulletin de la Soci{\'e}t{\'e} Royale des Sciences de Li{\`e}ge} 38: 86--95.

\bibitem[Speed(1969b)]{spe}
Speed, T. P. 1969. ``Some remarks on a class of distributive lattices." \emph{Jour. Aust. Math. Soc.} 9: 289--296.

\bibitem[Speed(1974)]{spe2}
Speed, T. P. 1974. ``Spaces of ideals of distributive lattices II, minimal
prime ideals." \emph{Jour. Aust. Math. Soc.} 18:54–-72.
\bibitem[Strauss(1967)]{stra}
Strauss, D. P. 1967. ``Extremally disconnected spaces." \emph{Proceedings of the American Mathematical Society} 18(2): 305--309.
\bibitem[Stone(1937)]{sto}
Stone M. H. 1937. ``Topological representations of distributive lattices and Brouwerian logics." \emph{{\v{C}}asopis pro p{\v{e}}siov{\'a}n{\'\i} matematiky a fysiky} 67:1--25.
\bibitem[Thakare and Nimbhorkar(1983)]{tha}
Thakare, N. K., and S. K. Nimbhorkar. 1983. ``Space of minimal prime ideals of a ring without nilpotent elements." \emph{Journal of Pure and Applied Algebra} 27(1): 75--85.
\bibitem[Varlet(1968)]{val}
Varlet, J. C. 1968. ``A generalization of the notion of pseudo-complememtedness." \emph{Bull. Soc. Roy. Sc. Li$\grave{e}$ge} 37: 149--158.
\end{thebibliography}
\end{document}